\documentclass[11pt]{amsart}
\usepackage{amsmath,paralist,amssymb}
\usepackage{amsthm}
\usepackage{hyperref}
\usepackage{graphicx,subfigure}
\usepackage{tikz,color}
\usepackage[margin=1in]{geometry}
\usetikzlibrary{arrows,shapes,snakes,backgrounds} 
\usepackage{verbatim}

\usepackage[all]{xy}

\definecolor{DarkBlue}{rgb}{0,0,0.8} 
\definecolor{DarkGreen}{rgb}{0,0.5,0.0}

\newtheorem*{lemma*}{Lemma}

\newtheorem*{corollary*}{Corollary}
\newtheorem*{theorem*}{Theorem}
\newtheorem*{theorem1*}{Theorem \ref{sp}}
\newtheorem*{theorem2*}{Theorem \ref{mm}}
\newtheorem*{conjecture*}{Conjecture \ref{conjcryDC}}

\numberwithin{equation}{section}
\newtheorem{theorem}[equation]{Theorem}
\newtheorem{proposition}[equation]{Proposition}
\newtheorem{lemma}[equation]{Lemma}
\newtheorem{corollary}[equation]{Corollary}
\newtheorem{conjecture}[equation]{Conjecture} 
\theoremstyle{remark}
\newtheorem{remark}[equation]{Remark}
\newtheorem{example}[equation]{Example}
\newtheorem{examples}[equation]{Examples}

\theoremstyle{definition}
\newtheorem{definition}[equation]{Definition}

\def\ee{{\bf e}}
  \def\e{\epsilon}
   \def\vol{{\rm vol}}
 
\def\vv{{\rm v}}
\def\i{{\rm in }}
\def\fin{{\rm fin}}
\def\v{{\bf a}}
\def\aa{{\bf a}}
 \def\f_H{{\bf w}}
 \def\f{{\bf f}}
 \def\a{{\bf a}}
\def\R{\mathbb{R}}
\def\N{\mathbb{N}}
\def\Z{\mathbb{Z}}
\def\g{{\bf a}}
\def\inc{\textrm{inc}}
 \def\F{\mathcal{F}}
\def\O{\mathcal{O}}
\def\I{\mathcal{I}}

\begin{document}

\tikzstyle{w}=[label=right:$\textcolor{red}{\cdots}$] 
\tikzstyle{b}=[label=right:$\cdot\,\textcolor{red}{\cdot}\,\cdot$] 
\tikzstyle{bb}=[circle,draw=black!90,fill=black!100,thick,inner sep=1pt,minimum width=3pt] 
\tikzstyle{bb2}=[circle,draw=black!90,fill=black!100,thick,inner sep=1pt,minimum width=2pt] 
\tikzstyle{b2}=[label=right:$\cdots$] 
\tikzstyle{w2}=[]
\tikzstyle{vw}=[label=above:$\textcolor{red}{\vdots}$] 
\tikzstyle{vb}=[label=above:$\vdots$] 

\tikzstyle{level 1}=[level distance=3.5cm, sibling distance=3.5cm]
\tikzstyle{level 2}=[level distance=3.5cm, sibling distance=2cm]

\tikzstyle{bag} = [text width=4em, text centered]
\tikzstyle{end} = [circle, minimum width=3pt,fill, inner sep=0pt]

\title[Flow polytopes of signed graphs and the Kostant partition function]{Flow polytopes of signed graphs and the Kostant partition function}
\author{Karola M\'esz\'aros}
\address{Karola M\'esz\'aros, Department of Mathematics, Cornell University, Ithaca NY 14853  \newline karola@math.cornell.edu
}

\author{Alejandro H.  Morales}
\address{Alejandro H. Morales,
Department of Mathematics, Massachusetts Institute of Technology, Cambridge, MA 02139 \newline ahmorales@math.mit.edu
}

\thanks{M\'esz\'aros is supported by a National Science Foundation Postdoctoral Research Fellowship.}
\date{\today}

\begin{abstract} We establish the relationship between volumes of flow polytopes associated to signed graphs and the Kostant partition function.  A special case of this relationship, namely, when the graphs are signless,  has been studied  in detail by Baldoni and Vergne using techniques of residues. In contrast with their approach, we provide  entirely combinatorial proofs inspired by the work of Postnikov and Stanley on flow polytopes. As a fascinating special family of flow polytopes, we  study  the Chan-Robbins-Yuen polytopes. Motivated by the beautiful volume formula $\prod_{k=1}^{n-2} Cat(k)$ for the type $A_n$ version,  where $Cat(k)$ is the $k$th Catalan number, we introduce  type $C_{n+1}$ and $D_{n+1}$ Chan-Robbins-Yuen polytopes along with intriguing conjectures pertaining to  their properties. \end{abstract}

\maketitle
  
\tableofcontents

  \section{Introduction}
\label{sec:flowsin}

In this paper we use combinatorial techniques to establish the relationship between volumes of flow polytopes associated to signed graphs and the Kostant partition function. Our techniques yield a systematic method for computing volumes of flow polytopes associated to signed graphs. We study special families of polytopes in detail, such as the Chan-Robbins-Yuen polytope \cite{CRY} and certain type $C_{n+1}$ and $D_{n+1}$ analogues of it. We also give several intriguing  conjectures for their volume.

Our results on  flow polytopes associated to signed graphs and the Kostant partition function   specialize to the results of Baldoni and Vergne, in which they established the connection between type $A_n$ flow polytopes and the Kostant partition function  \cite{BV2, BV1}. Baldoni and Vergne use residue techniques, while in their unpublished work Postnikov and Stanley took a combinatorial approach \cite{p,S}. In our study of type $A_n$ as well as type $C_{n+1}$ and  $D_{n+1}$ flow polytopes we establish the above mentioned connections by entirely combinatorial methods.

  Traditionally, flow polytopes are associated to loopless (and signless) graphs in the following way. Let $G$ be a graph on the vertex set $[n+1]$, and let $\i(e)$ denote the smallest (initial) vertex of edge $e$ and $\fin(e)$ the biggest (final) vertex of edge $e$.  Think of fluid flowing on the edges of $G$ from the smaller to the bigger vertices, so that the total fluid volume entering vertex $1$ is one and leaving vertex $n+1$ is one, and there is conservation of fluid at the intermediate vertices. Formally, a \textbf{flow} $f$ of size one on $G$ is a function $f: E \rightarrow \R_{\geq 0}$ from the edge set $E$ of $G$ to the set of nonnegative real numbers such that 
  
  $$1=\sum_{e \in E(G), \i(e)=1}f(e)= \sum_{e \in E(G),  \fin(e)=n+1}f(e),$$
  
\noindent  and for $2\leq i\leq n$
  
  $$\sum_{e \in E(G), \fin(e)=i}f(e)= \sum_{e \in E(G), \i(e)=i}f(e).$$
  
  \medskip

  The \textbf{flow polytope} $\F_G$ associated to the graph $G$ is the set of all flows $f: E \rightarrow \R_{\geq 0}$ of size one.
  A fascinating example  is the flow polytope $\F_{K_{n+1}}$ of the complete graph $K_{n+1}$, which is also called  the Chan-Robbins-Yuen polytope $CRYA_n$ \cite{CRY} (Chan, Robbins and Yuen defined it in terms of matrices),  and has kept the combinatorial community in its magic grip since its volume is equal to $\prod_{k=0}^{n-2}Cat(k),$ where  $Cat(k)=\frac{1}{k+1}{2k \choose k}$ is the $k$th Catalan number. This was proved analytically by Zeilberger \cite{Z}, but there is no combinatorial proof for this volume formula.

In their unpublished work \cite{p,S} Postnikov and Stanley  discovered the following  remarkable connection between the volume of the flow polytope and the Kostant partition function $K_G$:
 
 \begin{theorem1*}[\cite{p,S}] Given a loopless (signless) connected graph $G$ on the vertex set $[n+1]$, let $d_i=indeg_G(i)-1$, for $i \in \{2, \ldots, n\}$. Then, the normalized volume $\vol(\F_G)$ of the flow polytope $\F_G$ associated to the graph $G$ is 
 \begin{equation} \label{sp1}
 \vol(\F_G)=K_G(0,d_2, \ldots, d_n, -\sum_{i=2}^n d_i).
\end{equation}
  \end{theorem1*} 
  
  The notation $indeg_G(i)$ stands for the indegree of vertex $i$ in the graph $G$ and $K_G$ denotes the Kostant partition function associated to graph $G$. 
  
  In light of Theorem \ref{sp}, Zeilberger's result about the volume of the Chan-Robbins-Yuen polytope $CRYA_n$ can be stated as:

\begin{equation} \label{a}K_{K_{n-1}}(1, 2, \ldots, n-2, -\textstyle{{n-1 \choose 2}})=\prod_{k=1}^{n-2}Cat(k).\end{equation}

  Recall that the {\bf Kostant partition function}  $K_G$ evaluated at the vector $\v \in \Z^{n+1}$ is defined as

\begin{equation} \label{kost} K_G(\v)= \#  \{ (b_{k})_{k \in [N]} \mid \sum_{k \in [N]} b_{k}  \a_k =\v \textrm{ and } b_{k} \in \Z_{\geq 0} \},\end{equation}

\noindent where $[N]=\{1,2,\ldots,N\}$ and  $\{\{\a_1, \ldots, \a_N\}\}$ is the multiset of vectors corresponding to the multiset of edges of $G$ under the correspondence which associates an edge  $(i, j)$, $i<j$, of $G$ with a positive type $A_n$ root $\ee_i-\ee_j$, where $\ee_i$ is the $i$th standard basis vector in $\mathbb{R}^{n+1}$.

In other words, $K_G(\v)$ is the number of ways to write the vector $\v$ as a $\mathbb{N}$-linear combination of the positive type $A_n$ roots (with possible multiplicities) corresponding to the edges of $G$, without regard to order.  Note that for $K_G(\v)$ to be nonzero, the partial sums of the coordinates of $\v$ have to satisfy $a_1+\ldots+a_i \geq 0$, $i\in [n]$, and $a_1+\ldots+a_{n+1}=0$. Also, $K_G({\bf a})$ has the following formal generating series:
\begin{equation}\label{gsKA}
\sum_{\g \in \Z^{n+1}} K_G({\bf a}) x_1^{a_1}\cdots x_{n+1}^{a_{n+1}} = \prod_{(i,j)\in E(G)} (1-x_ix_j^{-1})^{-1}.
\end{equation} 
While endowed with combinatorial meaning, Kostant partition functions were introduced in and are a vital part of representation theory. For instance for classical Lie algebras, weight multiplicities and tensor product multiplicities (e.g., Littlewood-Richardson coefficients) can be expressed in terms of the Kostant partition function (see \cite{CC,GJH} and Steinberg's formula in \cite[Sec. 24.4]{JH}). Kostant partition functions also come up in toric geometry and approximation theory. A salient feature of $K_G({\bf a})$ is that it is a {\em piecewise quasipolynomial function} in ${\bf a}$ if $G$ is fixed \cite{DM,Str}.


We generalize Theorem \ref{sp} to establish the connection between flow polytopes associated to  loopless {\em signed} graphs and a {\bf dynamic} Kostant partition function $K_G^{\text{dyn}}({\bf a})$ with the following formal generating series: 
\begin{equation} \label{gsKC}
\sum_{\g \in \Z^{n+1}} K_{G}^{\text{dyn}}({\bf a}) x_1^{a_1}\cdots x_{n+1}^{a_{n+1}} =\prod_{(i,j,-) \in E(G)} (1-x_ix_j^{-1})^{-1} \prod_{(i,j,+) \in E(G)} (1-x_i-x_j)^{-1},
\end{equation}
where $G$ is a signed graph. By a signed graph we mean a graph where each edge has a positive or a negative sign associated to it. A signless graph can be thought of as a signed graph where all edges have a negative sign associated to them. The definition of a flow polytope associated to a signed graph generalizes the case of flow polytopes associated to signless graphs and can be found in Section \ref{sec:flowsdefs}.

We develop  a systematic method for calculating volumes of flow polytopes of signed graphs. There are several ways to state and specialize  our results; we highlight the next theorem as perhaps the most appealing special case. 

  \begin{theorem2*}  Given a loopless connected signed graph $G$ on the vertex set $[n+1]$, let $d_i=indeg_G(i)-1$ for $i\in \{2,\ldots,n\}$, where $indeg_G(i)$ is the indegree of vertex $i$. The normalized volume $\vol(\F_G)$ of the flow polytope $\F_G$ associated to graph $G$ is 
  
 $$\vol(\F_G)=K^{\text{dyn}}_{G}(0,d_2,\ldots,d_n,d_{n+1}),$$
  
\noindent  where $K^{\text{dyn}}_G$ has the generating series given in Equation \eqref{gsKC}. 
  \end{theorem2*} 




 Inspired by the intriguing $CRYA_n$ polytope, we introduce its type $C_{n+1}$ and $D_{n+1}$ analogues, $CRYC_{n+1}$ and $CRYD_{n+1}$, prove that their number of vertices are $3^{n}$ and $3^{n}-2^{n}$, respectively and we  conjecture the following.
 
 \begin{conjecture*} 
The normalized volumes of the type $C$ and type $D$ analogues $CRYC_{n+1}$ and $CRYD_{n+1}$  of the Chan-Robbins-Yuen polytope $CRYA_n$ are
\begin{align*}
\vol(CRYC_{n+1}) &= 2^{(n-1)^2+n} \prod_{k=0}^{n-1} Cat(k),\\
\vol(CRYD_{n+1}) &= 2^{(n-1)^2} \prod_{k=0}^{n-1} Cat(k),
\end{align*}
where $Cat(k)=\frac{1}{k+1}\binom{2k}{k}$ is the $k$th Catalan number.
\end{conjecture*}
 

\medskip

\noindent {\bf Outline:} 
In the first part of this paper we introduce  flow polytopes associated to signed graphs and 
characterize their vertices.
 In Section \ref{sec:flowsdefs} the necessary background on  signed graphs, Kostant partition functions and flows is given. We also define  flow polytopes associated to signed graphs  and remark that their Ehrhart functions can be  expressed in terms of Kostant partition functions. In Section \ref{sec:flowsvert} we give a characterization of the vertices of flow polytopes associated to signed graphs, and prove that the vertices of a special family of flow polytopes associated to signed graphs are integral, noting that in general this is not the case. As an application of the results from this section we find nice formulas for the number of vertices of the type $C_n$ and $D_n$ generalizations of the Chan-Robbins-Yuen polytope. 

The second part of the paper is about subdivisions of flow polytopes. In Section \ref{sec:red} we show that certain operations on graphs, called reduction rules, are a way of encoding subdivisions of flow polytopes. Using the reduction rules, in Section \ref{sec:flowssubdiv} we state and prove the Subdivision Lemma, which is a key ingredient of our subsequent explorations. The Subdivision Lemma gives a hands-on way of subdividing, and eventually triangulating, flow polytopes. 

The last part of the paper is about using the subdivision of flow polytopes to compute their volumes. In Section \ref{sec:vol} we use the Subdivision Lemma to prove Theorems \ref{sp} and \ref{mm}: namely that the volume of a flow polytope is equal to a value of the dynamic Kostant partition function. To do the above, we introduce the dynamic Kostant partition function in this section. The dynamic Kostant partition function specializes to the Kostant partition function in the case of signless graphs and has a nice and simple generating function, just like the Kostant partition function. We apply the above results in Section \ref{CRYAD} to the study of volumes of the Chan-Robbins-Yuen polytope and its various generalizations. We conclude our chapter  with several intriguing conjectures on the  volumes of the type $C_n$ and $D_n$ generalizations of the Chan-Robbins-Yuen polytope.

Supplementary code for calculating the volume of flow polytopes and for evaluating the (dynamic) Kostant partition function is available at the site: 
\begin{center}
\url{http://sites.google.com/site/flowpolytopes/}
\end{center}

\subsection*{Acknowledgements} 

We thank Alexander Postnikov and  Richard Stanley for encouraging us to work on this problem, and for discussions. We also thank Federico Ardila for his enthusiasm for our work and several suggestions and Olivier Bernardi for helpful comments.

 \section{Signed graphs, Kostant partition functions, and flows}
\label{sec:flowsdefs}

In this section we define the concepts of graphs, Kostant partition functions and flows, all in the \textbf{signed} universe. One can think of these as the generalization of these concepts' signless counterparts from the type $A_n$ (signless) root system to other types, such as $C_{n+1}$ and $D_{n+1}$. We also define general flow polytopes, which are a main object of this paper.  We conclude the section by giving simple properties of these polytopes and giving examples of the main flow polytopes we study.

Throughout this section,  the graphs $G$ on the vertex set $[n+1]$ that we consider are signed, that is there is a sign $\e \in \{+, -\}$ assigned to each of its edges. We allow loops and multiple edges. The sign of a loop is always $+$, and a loop at vertex $i$ is denoted by $(i, i, +)$. Denote by $(i, j, -)$ and $(i, j, +)$, $i < j$, a negative and a positive edge between vertices $i$ and $j$, respectively.  A positive edge, that is an edge labeled by $+$, is {\bf positively incident}, or, {\bf incident with a positive sign}, to both of its endpoints.  A negative  edge is positively incident to its smaller vertex and  {\bf negatively  incident} to  its greater endpoint. See Figure~\ref{turns} for an example of the incidences. 
Denote by $m_{ij}^\e$ the multiplicity of edge $(i, j, \e)$ in $G$, $i\leq j$, $\e \in \{+, -\}$. 
To each edge $(i, j, \e)$, $i\leq j$,  of $G$,  associate the positive type $C_{n+1}$ root $\vv(i,j, \e)$, where $\vv(i,j, -)=\ee_i-\ee_j$ and $\vv(i,j, +)=\ee_i+\ee_j$. Let $\{\{\a_1, \ldots, \a_N\}\}$ be the multiset of vectors corresponding to the multiset of edges of $G$ (i.e., $\a_k=v(e_k)$). Note that $N=\sum_{1\leq i\leq j\leq n+1} (m_{ij}^-+m_{ij}^+)$.

 The {\bf Kostant partition function}  $K_G$ evaluated at the vector $\v \in \Z^{n+1}$ is defined as

$$K_G(\v)= \# \{ (b_{k})_{k \in [N]} \mid \sum_{k \in [N]} b_{k}  \a_k =\v \textrm{ and } b_{k} \in \Z_{\geq 0}\}.$$

That is, $K_G(\v)$ is the number of ways to write the vector $\v$ as an $\mathbb{N}$-linear combination of the positive type $C_{n+1}$ roots corresponding to the edges of $G$, without regard to order.

 \begin{example}
For the signed graph $G$ in Figure \ref{AAA}, $K_G(1,3,-2)=3,$ since $(1,3,-2)=(\ee_1-\ee_3) + (2\ee_2) + (\ee_2-\ee_3)=(\ee_1+\ee_2) + 2(\ee_2-\ee_3)=(\ee_1-\ee_2)+(2\ee_2)+2(\ee_2-\ee_3)$.
\end{example}

\begin{figure}
\begin{center}
\subfigure[]{
\includegraphics[width=10cm]{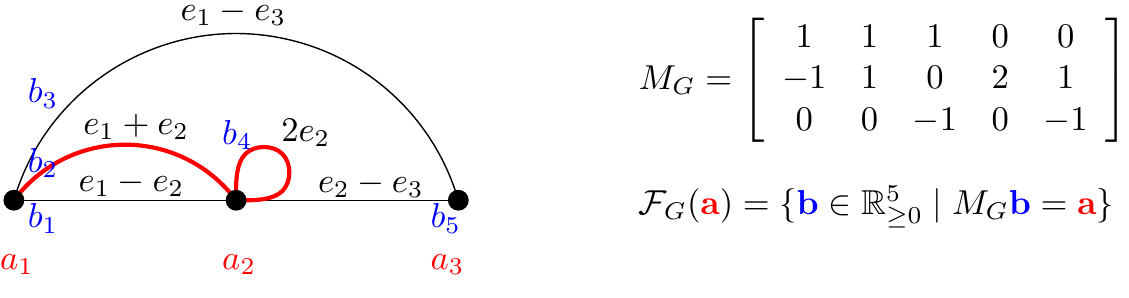}
\label{AAA}
}
\quad
\subfigure[]{
\includegraphics[width=4cm]{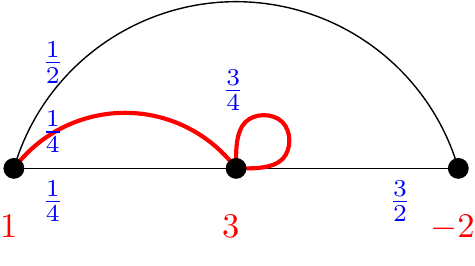}
\label{exflow}
}
\caption{ (a)
A signed graph $G$ on three vertices and the positive roots associated with each of the five edges. The columns of the matrix $M_G$ correspond to these roots.  The flow polytope $\F_G(\g)$ consists of the flows ${\bf b}\in \R_{\geq 0}^4$ such that $M_G{\bf b} = {\bf a}$ where ${\bf a}$ is the netflow vector. The Kostant partition function $K_{G}({\bf a})$ counts the lattice points of $\F_G(\g)$, the number of ways of obtaining ${\bf a}$ as a $\mathbb{N}$-integer combination of the roots associated to $G$.\newline
(b) A nonnegative flow on $G$ with  {netflow vector} ${\bf a}=(1, 3,-2)$. The flows on the edges are in \textcolor{blue}{blue}. Note that the total flow on the positive edges is $\textcolor{blue}{\frac14} + \textcolor{blue}{\frac34}=1=\frac12(\textcolor{red}{1}+\textcolor{red}{3}\textcolor{red}{-2})$.}
\label{figexthm1}
\end{center}
\end{figure}


Just like in the type $A_n$ case, we would like to think of the vector $(b_{i})_{i \in [N]} $ as a {\bf flow}. For this we here give a precise definition of flows in the type $C_{n+1}$ case, of which type $A_{n}$ is of course a special case.

Let $G$ be a signed graph on the vertex set $[n+1]$. Let $\{\{e_1, \ldots, e_N\}\}$ be the multiset of edges of $G$, and $S_G:=\{\{\a_1, \ldots, \a_N\}\}$  the multiset of  positive type $C_{n+1}$ roots   corresponding to the multiset of edges of $G$. Also, let $M_G$ be the $(n+1)\times N$ matrix whose columns are the vectors in $S_G$.  Fix an integer  vector $\g=(a_1, \ldots, a_n, a_{n+1}) \in \Z^{n+1}$.  

 An \textbf{$\g$-flow} $\f_G$ on $G$ is a vector $\f_G=(b_k)_{k \in [N]}$, $b_k \in \R_{\geq 0}$  such that $M_G \f_G=\g$. That is, for all $1\leq i \leq n+1$, we have 

 \begin{equation} \label{eqn:flow}
 \sum_{e \in E(G), \inc(e, v )=-} b(e)+a_v= \sum_{e \in E(G), \inc(e, v)=+} b(e)+\sum_{e=(v, v, + )} b(e),
  \end{equation}
  
\noindent   where $b(e_k)=b_k$, $\inc(e, v)=-$ if $e=(g, v, -)$, $g <v$, and $\inc(e, v)=+$ if $e=(g, v, +)$, $g <v$, or $e=( v, j, \e)$, $v <j,$ and $\e \in \{+, -\}$. 

\begin{example}
Figure \ref{exflow} shows a signed graph $G$ with three vertices with flow assigned to each edge. The netflow is ${\bf a}=(1,3,-2)$
\end{example}

Call $b(e)$  the \textbf{flow} assigned to edge $e$ of $G$. If the edge $e$ is negative, one can think of $b(e)$ units of fluid flowing on $e$ from its smaller to its bigger vertex. If the edge $e$ is positive, then one can think of $b(e)$ units of fluid flowing away both from $e$'s smaller and bigger vertex to ``infinity.'' Edge $e$ is then a ``leak" taking away $2b(e)$ units of fluid.

From the above explanation it is clear that if we are given an $\aa$-flow $\f_G$ such that 
\begin{equation} \label{eq:flowconstraint}
\sum_{i=1}^{n+1}a_i=2y,
\end{equation}
for some positive integer $y$ then  $\sum_{e=(i, j, +)}b(e)=y$.
 
 An {\bf  integer} \textbf{$\g$-flow} $\f_G$ on $G$ is an $\g$-flow  $\f_G=(b_i)_{i \in [N]}$, with $b_i \in \Z_{\geq 0}$. It is a   matter of checking the definitions to see that for a signed graph $G$ on the vertex set $[n+1]$ and vector  $\g=(a_1, \ldots, a_n, a_{n+1}) \in \Z^{n+1}$,  the number of  integer $\g$-flows on $G$ is given by the Kostant partition function, as highlighted in the next remark. 

\begin{remark}\label{flowsKG} Given a signed graph $G$ on the vertex set $[n+1]$ and a vector $\g=(a_1,\ldots,a_n,a_{n+1}) \in \Z^{n+1}$,  the integer $\g$-flows are in bijection with ways of writing $\g$ as a nonnegative linear combination of the roots associated to the edges of $G$. Thus $\#\{\text{integer } \g\text{-flows}\}=K_G(\g)$.
\end{remark}

Define the \textbf{flow polytope} $\F_G(\g)$ associated to a signed graph $G$ on the vertex set $[n+1]$ and the integer vector $\aa=(a_1, \ldots, a_{n+1})$ as the set of all $\g$-flows $\f_G$ on $G$, i.e., $\F_G=\{\f_G \in \R^N_{\geq 0} \mid M_G \f_G = \g\}$. The flow polytope 
  $\F_G(\g)$ then naturally lives in $\mathbb{R}^{N}$, where $N$ is the number  of edges of $G$.

Recall that $S_G$ denotes the multiset of $N$ vectors corresponding to the edges of $G$ and assume they span an $r$-dimensional space. Let $C(S_G)$ be the cone generated by the vectors in $S_G$. A vector $\g$ is in the interior of $C(S_G)$ if and only if $\g$ can be expressed as $\g=\sum_{i=1}^N b_i\a_i$ where $b_i>0$ for all $i$ \cite[Lemma 1.46.]{DCP}. If $\g$ is in the interior of $C(S_G)$, the dimension of $\F_G(\g)$ can be easily determined \cite[Sec. 1.1.]{BV2}.

\begin{proposition}[\cite{BV2}] \label{prop:dim}
The flow polytope $\F_G(\g)$ is empty if $\g\not\in C(S_G)$ and if $\g$ is in the interior of $C(S_G)$ then $\dim(\F_G(\g))=N-r$. This is also the dimension of the kernel of $M_G$.
\end{proposition}

\begin{remark} \label{rem:dim}
For a signed connected graph $G$ with vertex set $[n+1]$ and $N$ edges, if $\g$ is in the interior of $C(S_G)$, then $\dim(\F_G(\g))=\#E(G)-\#V(G)+1=N-n$ if $G$ only has negative edges (since $S_G$ spans the hyperplane $x_1+x_2+\cdots + x_{n+1}=0$), and $\dim(\F_G(\g))=\#E(G)-\#V(G)=N-n-1$ otherwise.
\end{remark}

Recall that given a  polytope $\mathcal{P}\subset \mathbb{R}^{N}$, the {$t^{th}$ dilate} of $\mathcal{P}$ is $\displaystyle t \mathcal{P}=\{(tx_1, \ldots, tx_{N}) \mid  (x_1, \ldots, x_{N}) \in \mathcal{P}\}.$ The number of lattice points of $t\mathcal{P}$, where $t$ is a nonnegative integer and $\mathcal{P}$ is a convex polytope, is given by the {\em Ehrhart function} $L_{\mathcal{P}}(t)$. If $\mathcal{P}$ has (rational) integral vertices then $L_{\mathcal{P}}(t)$ is a (quasi) polynomial (for background on the  theory of Ehrhart polynomials see \cite{BR}). From the definition of the Ehrhart function and the Kostant partition function it follows that
\begin{equation} \label{ehr} \displaystyle L_{\F_G(\g)} (t) = K_G(t \g). \end{equation}

Recall also that given two polytopes $\mathcal{P}_1$ and $\mathcal{P}_2$, their {\em Minkowski sum} is $\mathcal{P}_1+\mathcal{P}_2 = \{ v_1 + v_1 \mid v_1 \in \mathcal{P}_1, v_2 \in \mathcal{P}_2\}$. For a flow polytope $\F_G(\g)$ where $G$ is a signed graph on the vertex set $[n+1]$ and $\g\in \Z^{n+1}$, we have that  $\F_G(\g)$ is the Minkowski sum: 
\begin{equation}\label{eq:Minkowski}
\F_G(\g) = y\F_G(2\ee_1) +  \sum_{i=1}^{n}(a_i-2y\delta_{1,i})\F_{G}(\ee_i-\ee_{n+1}),
\end{equation}
 since $\g = y(2\ee_1) + \sum_{i=1}^{n}(a_i-2y\delta_{1,i})(\ee_i-\ee_{n+1})$,
where $\delta_{1,i}$ is the Kronecker delta and $2y=\sum_{i=1}^{n+1}a_i$. 

By~\eqref{eq:flowconstraint}, the flow polytopes $\F_G(\ee_i-\ee_{n+1})$ consists of flows with zero flow on the positive edges of $G$. Thus we can regard $\F_G(\ee_i-\ee_{n+1})$ as a type $A_n$ flow polytope on the signless graph obtained from $G$ by disregarding its positive edges (we can also ignore the negative edges $(a,b,-)$ $a<b<i$ since they also have zero flow). Such type $A_n$ flow polytopes have been widely studied \cite{BLV,BV2,DCP} and we  discuss their volumes in Section~\ref{vols}. The remaining polytope in the Minkowski sum~\eqref{eq:Minkowski}, $\F_G(2\ee_1)$, in general will consist of flows with one unit of flow on the positive edges. The volume of such type $D_{n+1}$ polytopes will be studied in Section~\ref{volsD}.

Finally, we give the main examples of the flow polytopes we study (see Figure~\ref{figmainexs}):
\begin{examples} \hfill \label{mainexs}
\begin{compactitem}
\item[(i)] Let $G$ be the graph with vertices $\{1,2\}$ and edges $(1,2,-)$ with multiplicity $m_{12}$; and let $\a=(1,-1)$. Then $\F_{G}(1,-1)$ is an $(m_{12}-1)$-dimensional simplex.
\item[(ii)] Let $G$ be the signed graph with one vertex $\{1\}$ and loops $(1,1,+)$ with multiplicity $m_{11}$; and let $\a=2$. Then $\F_G(2)$ is an $(m_{11}-1)$-dimensional simplex.
\item[(iii)] Let $G=K_{n+1}$ be the complete graph with $n+1$ vertices (all edges $(i,j,-)$ $1\leq i<j\leq n+1$) and $\a=\ee_1-\ee_{n+1}$. Then $\F_{K_{n+1}}(\ee_1-\ee_{n+1})$ is the type $A_n$ Chan-Robbins-Yuen polytope or $CRYA_n$ \cite{CR,CRY}. Such polytope is a face of the Birkhoff polytope of all $n\times n$ doubly stochastic matrices. It has dimension $\binom{n}{2}$, $2^{n-1}$ vertices, and Zeilberger \cite{Z} showed that its normalized volume is $\vol(CRYA_n)=\prod_{k=0}^{n-2}Cat(k)$ where $Cat(k)=\frac{1}{k+1}\binom{2k}{k}$ is the $k$th Catalan number. 
\item[(iv)] Let $G=K^D_n$ be the complete {\em signed} graph with $n$ vertices (all edges $(i,j,\pm)$ $1\leq i<j\leq n$) and $\a=2\ee_1$. Then $CRYD_n=\F_{K^D_n}(2\ee_1)$ is a type $D_n$ analogue of $CRYA_n$. We show it is integral (see Theorem~\ref{vertices2}) with dimension $n(n-2)$ and $3^{n-1}-2^{n-1}$ vertices (see Proposition~\ref{cryd}). We conjecture (see Conjecture~\ref{conjcryDC}) that its normalized volume is $2^{(n-2)^2}\cdot \vol(CRYA_n)$.
\item[(v)] Let $G=K^C_n$ be the complete {\em signed} graph with $n$ vertices and with loops $(i,i,+)$ corresponding to the type $C$ positive roots $2\ee_i$. Then $CRYC_n=\F_{K^C_n}(2\ee_1)$ is a type $C_n$ analogue of $CRYA_n$.  We show it is integral (see Theorem~\ref{vertices2}) with dimension $n(n-2)$ and $3^{n-1}$ vertices (see Proposition~\ref{cryc}). We conjecture (see Conjecture~\ref{conjcryDC}) that its normalized volume is $2^{n-1}\cdot \vol(CRYD_n)$.
\end{compactitem}
\end{examples}

\begin{figure}
\centering
\includegraphics[width=14cm]{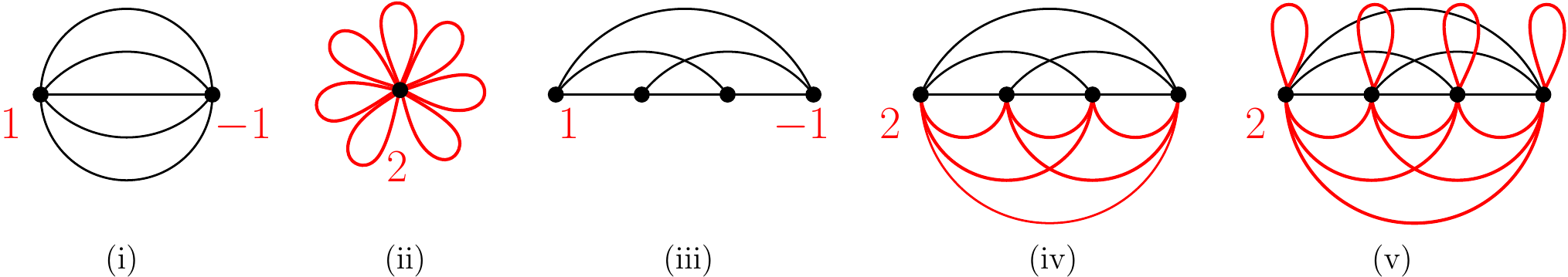}
\caption{Graphs and \textcolor{red}{netflow} whose flow polytopes are: (i), (ii) simplices and (iii),(iv),(v) instances of $CRYA_n$, $CRYD_n$ and $CRYC_n$.} 
\label{figmainexs}
\end{figure}






\section{The vertices of the flow polytope $\F_G(\g)$} \label{sec:flowsvert}

\subsection{Vertices of $\F_G(\g)$}
In this section we characterize the vertices of the flow polytope $\F_G(\g)$. Remarkably, if $G$ is a graph with only negative edges,  then   for any integer vector $\g$ the vertices of $\F_G(\g)$ are integer. Such a statement  is not true for signed graphs $G$ in general.  However, we show, using our characterization of the vertices of $\F_G(\g)$  that for special integer vectors $\g$ the vertices of $\F_G(\g)$  are integer. As an application of our vertex characterization, we show that the numbers of vertices of the type $C_{n+1}$ and type $D_{n+1}$ analogues of the Chan-Robbins-Yuen polytope from Examples~\ref{mainexs} (iv),(v) are $3^n$ and $3^n-2^n$, respectively.

That the vertices of $\F_G(\g)$ are integer for any signless graph $G$ and any integer vector $\g$ follows from the fact that the matrix $M_G$, whose columns are the positive type $A$ roots associated to the edges of $G$, is totally unimodular. However, as mentioned above, for signed graphs $G$ the polytope $\F_G(\g)$ does not always have integer vertices as the following simple example shows.

\begin{example}
Let $G$ be the graph 
\includegraphics[height=0.45cm]{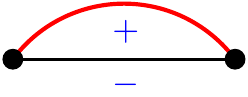} the flow polytope $\F_G(1,0)$ is a zero dimensional polytope with a vertex $(1/2,1/2)$. 
\end{example}

In the rest of the section $G$ denotes a signed graph. Recall that we defined $\g$-flows to be nonnegative. In this section we use the term \textbf{nonzero signed ${\bf 0}$-flow} to refer to a flow where we allow flows to be negative or positive or zero (as signified by signed),  which is not zero everywhere (signified by nonzero) and where the netflow is ${\bf 0}$.

\begin{lemma} \label{def}
An $\g$-flow $\f_G$ on $G$ is a vertex of  $\F_G(\g)$ if and only if there is no nonzero signed ${\bf 0}$-flow $\f^{\e}_G$ such that  $\f^{}_G-\f^{\e}_G$ and $\f^{}_G+\f^{\e}_G$ are flows on $G$.
\end{lemma}

Lemma \ref{def} follows from definitions, but since it is the starting point of the characterization of the vertices of $\F_G(\g)$,  we include a proof for clarity.

\medskip

\noindent \textit{Proof of Lemma \ref{def}.}
If there is a nonzero signed ${\bf 0}$-flow $\f^{\e}_G$ such that  $\f^{}_G-\f^{\e}_G$ and $\f^{}_G+\f^{\e}_G$ are flows  on $G$ (and thus $\g$-flows on $G$), then  $$\f_G=((\f^{}_G-\f^{\e}_G)+(\f^{}_G+\f^{\e}_G))/2,$$
so $\f_G$ is not a vertex of  $\F_G(\g)$.

If $\f_G$ is not a vertex of  $\F_G(\g)$, then  $\f_G$ can be written as 
  $$\f_G=(\f^{1}_G+\f^{2}_G)/2,$$ for some $\g$-flows $\f^{1}_G$ and $\f^{2}_G$ on $G$. Thus,  $$\f^{1}_G=\f_G-\f^{\e}_G$$ and $$\f^{2}_G=\f_G+\f^{\e}_G,$$ for some nonzero signed ${\bf 0}$-flow $\f^{\e}_G$.
\qed

\begin{lemma} \label{def2}
There is a nonzero signed ${\bf 0}$-flow $\f^{\e}_G$ such that  $\f^{}_G-\f^{\e}_G$ and $\f^{}_G+\f^{\e}_G$ are flows on $G$ if and only if there is a nonzero signed ${\bf 0}$-flow $\f^{\e}_G$ on $G$ whose support is contained in the support of $\f_G$. 
\end{lemma}

\proof One implication is trivial, and the other one follows by observing that given a nonzero signed ${\bf 0}$-flow $\f^{\e}_G$ on $G$ whose support is contained in the support of $\f_G$, we can obtain another nonzero signed ${\bf 0}$-flow $\f^{\e'}_G$ on $G$  whose support is contained in the support of $\f_G$ such that the  absolute value of the values $\f^{\e'}_G(e),$ for edges $e \in G$, is arbitrarily small, by simply letting $\f^{\e'}_G=\f^{\e}_G/M,$ for some large value of $M$. Thus if there is a nonzero signed ${\bf 0}$-flow $\f^{\e}_G$ on $G$ whose support is contained in the support of $\f_G$, then we can construct a nonzero signed ${\bf 0}$-flow $\f^{\e'}_G$ such that  $\f^{}_G-\f^{\e'}_G$ and $\f^{}_G+\f^{\e'}_G$ are flows on $G$. \qed

\begin{corollary} \label{cor:def}
An $\g$-flow $\f_G$ on $G$ is a vertex of  $\F_G(\g)$ if and only if there is no nonzero signed ${\bf 0}$-flow $\f^{\e}_G$ on $G$ whose support is contained in the support of $\f_G$.
\end{corollary}

\proof Corollary \ref{cor:def} follows from  Lemmas \ref{def} and \ref{def2}.\qed

\begin{lemma} \label{0}
If $H \subset G$ is the support of a  nonzero signed ${\bf 0}$-flow $\f^{\e}_G$, then $H$ contains no vertices of degree 1.
 \end{lemma}

\proof If $H$ contained a degree $1$ vertex, $\f^{\e}_G$ with support $H$ could not be a ${\bf 0}$-flow.
\qed

\medskip 
A \textbf {cycle} $C$ is a sequence of oriented edges $e_1, \ldots, e_k$ such that the second vertex of $e_i$ is the first vertex of $e_{i+1}$ for $i \in [k]$ and with $k+1$ identified with $1$. The number of \textbf{turns} in $C$ is the number of times two consecutive edges meet at a vertex of $C$ such that the edges of $C$ are incident with the same sign to that vertex (repetition of vertices allowed). A cycle $C$ of the graph $G$ is called \textbf{even} if  it has an even number of turns and \textbf{odd} otherwise. 
See Figure \ref{turns}.

\begin{figure} 
\centering
\subfigure[]{
\includegraphics[width=4.5cm]{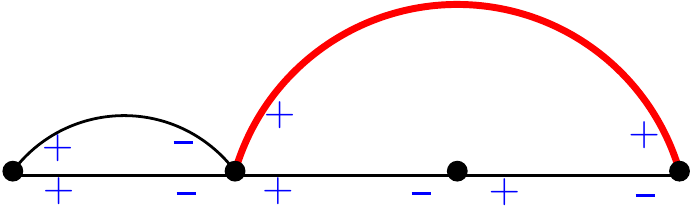}
\label{exslide2}
}
\qquad
\subfigure[]{
\includegraphics[width=3cm]{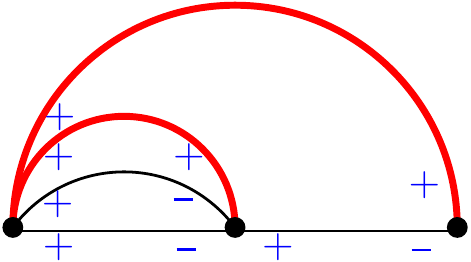}
\label{exslide3}
}
\caption{Regardless of how we order the edges above to form a cycle, the number of turns in the cycle will be $1$ in \subref{exslide2} and even in \subref{exslide3}. Thus, the resulting cycle in \subref{exslide2} is odd and in \subref{exslide3} is even.}
\label{turns}
\end{figure}

\begin{lemma} \label{order}
Given a set of edges which can be ordered to yield a cycle $C$, the parity of the number of turns of $C$ is the same as that of any other cycle that the edges can be ordered to give.
\end{lemma}

We leave the proof of Lemma \ref{order} as an exercise to the reader. For examples see Figures \ref{exslide2} and \ref{exslide3}.

\begin{lemma} \label{cycle}
If $H \subset G$ is the support of a  nonzero signed ${\bf 0}$-flow $\f^{\e}_G$, then $H$ contains an even cycle. \end{lemma}

\proof Since by Lemma \ref{0} $H$ contains no vertices of degree 1, each edge of $H$ is contained in at least one cycle. Let $k$ be the number of linearly independent cycles in $H$ in the binary cycle space. If $k=1$ and the nonzero signed ${\bf 0}$-flow $\f^{\e}_G$ has support $H$, then it follows by inspection that $H$ is an even cycle. If $k>1$ and the nonzero signed ${\bf 0}$-flow $\f^{\e}_G$ has support $H$, let   $P \subset H$ be a path such that $H-P$ contains $k-1$ linearly independent cycles and no vertices of degree $1$. If $P$ is contained in an even cycle in $H$, then we are done. If $P$ is not contained in an even cycle of $H$, then there are two paths $C_1$ and $C_2$ in $H$ such that $P+C_1$ and $P+C_2$ are cycles, but not even. Inspection shows that the  cycle $C_1+C_2$ is even.
\qed

\begin{lemma} \label{proper} 
If $C\subset G$ is an even cycle, then there exists a  nonzero signed ${\bf 0}$-flow $\f^{\e}_G$ with support $C$.
\end{lemma}

\proof Set  $\f^{\e}_G(e)=0$ for $e \in G-C$ and  $\f^{\e}_G(e)\in \{+\e, -\e\}$ for $e \in C$. Note that since $C$ is even there will be two such  nonzero signed ${\bf 0}$-flows $\f^{\e}_G$.
\qed

\begin{lemma} \label{proper2} 
There is a nonzero signed ${\bf 0}$-flow $\f^{\e}_G$ on $G$ whose support is contained in the support of the $\g$-flow $\f_G$ if and only if  the support of  $\f_G$  contains an even cycle.
\end{lemma}

\proof By Lemma \ref{cycle} if there is a nonzero signed ${\bf 0}$-flow $\f^{\e}_G$ on $G$ with support $H$, then $H$ contains an even cycle. Thus, in particular, if there is a nonzero signed ${\bf 0}$-flow $\f^{\e}_G$ on $G$ whose support is contained in the support of $\f_G$, then the support of  $\f_G$  contains an even cycle. Conversely, 
by Lemma \ref{proper} if $C$ is an even cycle contained in the support of  $\f_G$, then there is a nonzero signed ${\bf 0}$-flow $\f^{\e}_G$ on $G$ whose support is $C$, and thus contained in the support of $\f_G$.
\qed

\begin{theorem} \label{vertices}
An $\g$-flow $\f_G$ on $G$ is a vertex of  $\F_G(\g)$ if and only if  the support of  $\f_G$  contains no even cycle.
\end{theorem}

\proof Corollary \ref{cor:def} and Lemma \ref{proper2} imply the statement of Theorem \ref{vertices}. 
 \qed

\begin{theorem} \label{vertices2} If $\g=(2, 0, \ldots, 0)$, then the vertices of $\F_G(\g)$ are integer. In particular, the set of  vertices of $\F_G(\g)$ is a subset  of the set of integer $\g$-flows on $G$.   
\end{theorem}

By Theorem \ref{vertices}, in order to prove Theorem \ref{vertices2}, it suffices to show that if  the support of the $(2, 0, \ldots, 0)$-flow $\f_G$  contains no even cycle, then $\f_G$ is an integer flow. To achieve this, we characterize all possible odd cycles with no even subcycles  in the support of  a  $(2, 0, \ldots, 0)$-flow $\f_G$. By a \textbf{subcycle} $C'$ of a cycle $C$ we mean a cycle $C'$ whose edges are a subset of the edges of $C$.

\begin{proposition} \label{even}
A cycle $C$ contained  in the support of  a  $(2, 0, \ldots, 0)$-flow $\f_G$   contains no even subcycles  if and only if its set of edges  is of one of the three following forms:

\begin{compactitem}

\item[(i)] $\{(v_1, v_2, -), \ldots, (v_{k-1}, v_{k}, -)\}\cup \{ (w_1, w_2, -), \ldots, (w_{l-1}, w_{l}, -)\}\cup \{ (w_l, v_{k}, +)\}$, where $v_1=w_1$, $2\leq k, l$ and $v_1, \ldots, v_k, w_2, \ldots, w_l$ are distinct. See Figure \ref{I}.

\item[(ii)]  $\{(v_1, v_2, -), \ldots, (v_{k-1}, v_{k}, -)\}\cup \{ (v_1, v_{k}, +)\}$, where $v_1, \ldots, v_k$ are distinct. See Figure \ref{II}.

\item[(iii)]  $\{(v_1, v_1, +)\}$

\end{compactitem}

\end{proposition}

\begin{figure} 
\centering
\subfigure[]{
\includegraphics[height=2cm]{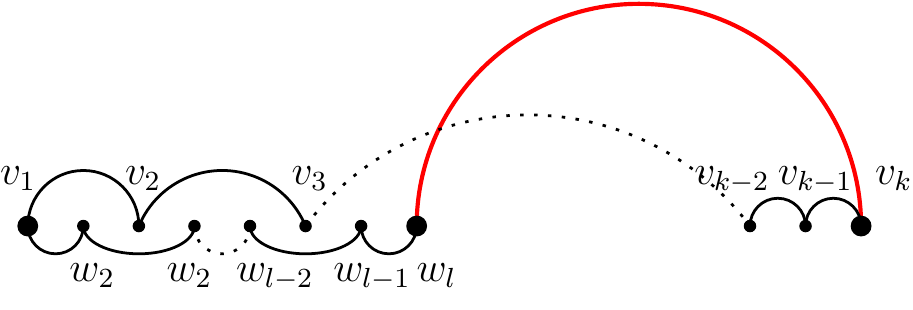}
\label{I}
}
\qquad 
\subfigure[]{
\includegraphics[height=2cm]{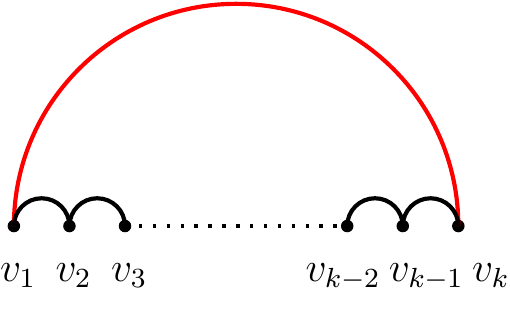}
\label{II}
}
\caption{Illustration of forms (i) and (ii) of Proposition~\ref{even}} 
\end{figure}

\proof One direction is trivial.

To prove the other direction, let $G'$ be the support of $f_G$. Observe that all vertices in $G'$ must have  a negative edge incident to them in order for the netflow to be ${\bf 0}$ at all but the first vertex, unless $G'$ is simply a loop at vertex $1$. Note that a cycle with only negative edges is even. 
Note that a path of negative edges (which is not a cycle) can be contracted without affecting the parity of the number of turns of a cycle. The above observations together are sufficient to prove the non-trivial direction of the proposition.
\qed
\bigskip

\noindent \textit{Proof of Theorem \ref{vertices2}.} Suppose that the $(2, 0, \ldots, 0)$-flow $\f_G$ is a vertex of  $\F_G(2, 0, \ldots, 0)$. Let $G'$ be the support of $f_G$.  Theorem \ref{vertices} and  Proposition \ref{even} imply that $G'$ contains exactly one  cycle $C$ which contains no even subcycle and whose smallest vertex is $v$. If $v=1$ then $G'=C$ and if $v>1$ then $G'$ is the union of $C$ and a path $(1, z_1, -), (z_1, z_2, -), \ldots, (z_m, v, -)$. In both cases it is evident that the flow $\f_G$ has to be integer in order to be a  $(2, 0, \ldots, 0)$-flow.
\qed

Note that the proof of Theorem \ref{vertices2}  characterizes all vertices of $\F_G(2, 0, \ldots, 0)$  very concretely. We summarize the results in Theorem \ref{vertices3}.

\begin{theorem} \label{vertices3}
A $(2, 0, \ldots, 0)$-flow $\f_G$ on $G$ is a vertex of  $\F_G((2, 0, \ldots, 0))$ if and only if  it is the unique integer $(2, 0, \ldots, 0)$-flow on $G$ with  support of one of the following forms:

\begin{compactitem}

\item[(i)] $\{(v_1, v_2, -), \ldots, (v_{k-1}, v_{k}, -)\}\cup \{ (w_1, w_2, -), \ldots, (w_{l-1}, w_{l}, -)\}\cup \{ (w_l, v_{k}, +)\}$, where $v_1=w_1=1$, $2\leq k, l$ and $v_1, \ldots, v_k, w_2, \ldots, w_l$ are distinct.

\item[(ii)] $\{(1, z_1, -), (z_1, z_2, -), \ldots, (z_m, v_1, -)\}\cup \{(v_1, v_2, -), \ldots, (v_{k-1}, v_{k}, -)\}\cup$\\
$\cup \{ (w_1, w_2, -), \ldots, (w_{l-1}, w_{l}, -)\}\cup \{ (w_l, v_{k}, +)\}$, where $v_1=w_1$, $2\leq k, l$ and $v_1, \ldots, v_k$, $w_2, \ldots, w_l$ are distinct. 

\item[(iii)]  $\{(v_1, v_2, -), \ldots, (v_{k-1}, v_{k}, -)\}\cup \{ (v_1, v_{k}, +)\}$, where $1=v_1, \ldots, v_k$ are distinct. 

\item[(iv)]  $\{(1, z_1, -), (z_1, z_2, -), \ldots, (z_m, v_1, -)\}\cup \{(v_1, v_2, -), \ldots, (v_{k-1}, v_{k}, -)\}\cup \{ (v_1, v_{k}, +)\}$, where $v_1, \ldots, v_k$ are distinct.

\item[(v)]  $\{(v_1, v_1, +)\}$

\item[(vi)]  $\{(1, z_1, -), (z_1, z_2, -), \ldots, (z_m, v_1, -)\}\cup \{(v_1, v_1, +)\}$

\end{compactitem}

\end{theorem}

\subsection{Vertices of the type $D_{n+1}$ and type $C_{n+1}$ Chan-Robbins-Yuen polytope}  Theorem \ref{vertices3} gives a hands-on characterization of the vertices of any type $D_{n+1}$ and type $C_{n+1}$ flow polytope. In this section we show how to use it to count the number of vertices of the  type $C_{n+1}$ and type $D_{n+1}$ Chan-Robbins-Yuen polytopes $CRYD_{n+1}$ and $CRYC_{n+1}$. 

Recall that the flow polytope $\F_{K_{n+1}}(1, 0, \ldots, 0, -1)$ of the complete graph $K_{n+1}$ from Examples~\ref{mainexs} (iii) is  the Chan-Robbins-Yuen polytope $CRYA_n$ \cite{CRY}.  One way to generalize $CRYA_n$ is to consider the complete signed graphs in type $C_{n+1}$ and type $D_{n+1}$ (see Examples~\ref{mainexs} (iv), (v)). 

Let $K^{D}_{n+1}$ be the complete signed graph on $n+1$ vertices  of type $D_{n+1}$ (all edges of the form $(i,j,\pm)$ for $1\leq i<j\leq n+1$ corresponding to all the positive roots $e_i \pm e_j$ for $1\leq i<j\leq n+1$ in type $D_{n+1}$). Then the polytope $CRYD_{n+1}=\mathcal{F}_{K_{n+1}^{D}}(2,0,\ldots,0)$ is an analogue of the Chan-Robbins-Yuen polytope. The vector $(2, 0, \ldots, 0)$ is the highest root of type $C_{n+1}$, and we pick this vector as opposed to the highest root of type $D_{n+1}$, because we would like the vertices of $CRYD_{n+1}$ to be integral. If we were to study $\mathcal{F}_{K_{n+1}^{D}}(1,1,0,\ldots,0)$, where $(1,1, 0, \ldots, 0)$ is the highest root of type $D_{n+1}$, the vertices of this polytope would not be integral (for example, two of the seventeen vertices of the flow polytope $\mathcal{F}_{K_4}^D(1,1,0,0)$ are rational).   Note that any signed graph on the vertex set $[n+1]$, including $K^{D}_{n+1}$, can be considered a type $C_{n+1}$ graph, so that the choice of the highest root of $C_{n+1}$ is not unnatural. 

Let $K_{n+1}^C$ be the complete signed graph together with loops $(i,i,+)$, $1\leq i\leq n+1$,  corresponding to the type $C_{n+1}$ positive roots $2\ee_i$ and let  $CRYC_{n+1}=\mathcal{F}_{K_{n+1}^{C}}(2,0,\ldots,0)$.

\begin{proposition} \label{cryd}  The polytope $CRYD_{n+1}$ has $3^n-2^n$ vertices.
\end{proposition}

\proof We prove the statement by induction. The base of induction is clear. Suppose that $CRYD_{n}$ has $3^{n-1}-2^{n-1}$ vertices. Using Theorem \ref{vertices3} we see that the vertices of $CRYD_{n+1}$ have to be the unique integer $(2, 0, \ldots, 0)$-flows on $G$ with  support of the form:

\begin{itemize}

\item $\{(1, i, -)\}\cup S(i, n+1)$, where $2\leq i \leq n$ and $S(i, n+1)$ is the support of a vertex of $CRYD_{n+2-i}$ where we consider the flow graph of  $CRYD_{n+2-i}$ to be on the vertex set $\{i, i+1, \ldots, n+1\}$.

\item $\{(v_1, v_2, -), \ldots, (v_{k-1}, v_{k}, -)\}\cup \{ (w_1, w_2, -), \ldots, (w_{l-1}, w_{l}, -)\}\cup \{ (w_l, v_{k}, +)\}$, where $v_1=w_1=1$, $2\leq k, l$ and $v_1, \ldots, v_k, w_2, \ldots, w_l$ are distinct.

\item  $\{(v_1, v_2, -), \ldots, (v_{k-1}, v_{k}, -)\}\cup \{ (v_1, v_{k}, +)\}$, where $1=v_1, \ldots, v_k$ are distinct. 

\end{itemize}

Call the supports of the above forms of type I, II and II, respectively.

By induction, the number of vertices of $CRYD_{n+1}$ of type I is $$\sum_{i=2}^n (3^{n+1-i}-2^{n+1-i}).$$ 

By inspection, the number of vertices of $CRYD_{n+1}$ of type II is $$\sum_{1<i<j\leq n+1} (3^{i-2}2^{j-i-1}).$$ 

Finally, the number of vertices of $CRYD_{n+1}$ of type III is $$\sum_{i=2}^{n+1} 2^{i-2}.$$ 

It is a matter of simple algebra to show that 

$$\sum_{i=2}^n (3^{n+1-i}-2^{n+1-i})+\sum_{1<i<j\leq n+1} (3^{i-2}2^{j-i-1})+\sum_{i=2}^{n+1} 2^{i-2}=3^n-2^n.$$
\qed

\begin{proposition} \label{cryc} The polytope $CRYC_{n+1}$ has $3^n$ vertices.
\end{proposition}

\proof Using Theorem \ref{vertices3} we see that the set of vertices of $CRYC_{n+1}$ is equal to the set of vertices of $CRYD_{n+1}$ together with the vertices which are the  unique integer $(2, 0, \ldots, 0)$-flows on $G$ with  support of the form:

\begin{itemize}

\item  $\{(v_1, v_1, +)\}$

\item  $\{(1, z_1, -), (z_1, z_2, -), \ldots, (z_m, v_1, -)\}\cup \{(v_1, v_1, +)\}$

\end{itemize}

By Proposition \ref{cryd} the number of vertices of $CRYD_{n+1}$ is $3^n-2^n$ and the number of vertices of the form described above is $2^n$. Thus, Proposition \ref{cryc} follows.

\qed

 \section{Reduction rules of the flow polytope $\F_G(\aa)$}

 \label{sec:red}

In this section we propose an algorithmic way of triangulating the flow polytope $\F_G(\aa)$.  This also yields a systematic way to calculate the volume of $\F_G(\aa)$ by summing the volumes of the simplices  in the triangulation. The process of triangulation of  $\F_G(\aa)$ is closely related to the triangulation of root polytopes by subdivision algebras, as studied by M\'esz\'aros in \cite{m2, m1}. 

Given a signed graph $G$ on the vertex set $[n+1]$, if we have two edges incident to vertex $i$ with opposite signs, e.g., $(a,i,-),(i,b,+)$ with flows $p$ and $q$, we  add a new edge not incident to $i$,  e.g., $(a,b,+)$, and discard one or both of the original edges to obtain graphs $G_1,G_2$, and $G_3$, respectively. We then reassign flows to preserve the original netflow on the vertices. We look at all possible cases and obtain the reduction rules (R1)-(R6) in Figure~\ref{fig:subdivrules}.



\subsection{Reduction rules for signed graphs} 
   
   Given   a graph $G$ on the vertex set $[n+1]$ and   $(a, i, -), (i, b, -) \in E(G)$ for some $a<i<b$, let   $G_1, G_2, G_3$ be graphs on the vertex set $[n+1]$ with edge sets
  \begin{align*} 
E(G_1)&=E(G)\backslash \{(a, i,-)\} \cup \{(a, b, -)\},  \\
E(G_2)&=E(G)\backslash \{(i, b,-)\} \cup \{(a, b,-)\}, \tag{R1} \\ 
E(G_3)&=E(G)\backslash \{(a, i,-)\} \backslash \{(i, b,-)\} \cup \{(a, b,-)\}. 
\end{align*}
    Given   a graph $G$ on the vertex set $[n+1]$ and   $(a, i, -), (i, b, +) \in E(G)$ for some $a<i<b$, let   $G_1, G_2, G_3$ be graphs on the vertex set $[n+1]$ with edge sets
  \begin{align*} 
E(G_1)&=E(G)\backslash \{(a, i,+)\} \cup \{(a, b, +)\},  \\
E(G_2)&=E(G)\backslash \{(i, b,-)\} \cup \{(a, b, +)\},\tag{R2} \\ 
E(G_3)&=E(G)\backslash \{(a, i,-)\} \backslash \{(i, b, +)\} \cup \{(a, b, +)\}. 
\end{align*}
   Given   a graph $G$ on the vertex set $[n+1]$ and   $(a, i, -), (b, i, +) \in E(G)$ for some $a<b<i$, let   $G_1, G_2, G_3$ be graphs on the vertex set $[n+1]$ with edge sets
  \begin{align*} 
E(G_1)&=E(G)\backslash \{(a, i, -)\} \cup \{(a, b, +)\},  \\
E(G_2)&=E(G)\backslash \{(b, i, +)\} \cup \{(a, b, +)\},\tag{R3} \\ 
E(G_3)&=E(G)\backslash \{(a, i, -)\} \backslash \{(b, i, +)\} \cup \{(a, b, +)\}. 
\end{align*}
   Given   a graph $G$ on the vertex set $[n+1]$ and   $(a, i, +), (b, i, -) \in E(G)$ for some $a<b<i$, let   $G_1, G_2, G_3$ be graphs on the vertex set $[n+1]$ with edge sets
  \begin{align*} 
E(G_1)&=E(G)\backslash \{(a, i, +)\} \cup \{(a, b, +)\},  \\
E(G_2)&=E(G)\backslash \{(b, i, -)\} \cup \{(a, b, +)\}, \tag{R4} \\ 
E(G_3)&=E(G)\backslash \{(a, i, +)\} \backslash \{(b, i, -)\} \cup \{(a, b, +)\}. 
\end{align*}

 Given   a graph $G$ on the vertex set $[n+1]$ and   $(a, i, -), (a, i, +) \in E(G)$ for some $a<i $, let   $G_1, G_2, G_3$ be graphs on the vertex set $[n+1]$ with edge sets
  \begin{align*} 
E(G_1)&=E(G)\backslash \{(a, i, +)\} \cup \{(a, a, +)\},  \\
E(G_2)&=E(G)\backslash \{(a, i, -)\} \cup \{(a, a, +)\},\tag{R5} \\ 
E(G_3)&=E(G)\backslash \{(a, i, +)\} \backslash \{(a, i, +)\}   \cup \{(a, a, +)\}. 
\end{align*}

Given   a graph $G$ on the vertex set $[n+1]$ and   $(a, i, -), (i, i, +) \in E(G)$ for some $a<i $, let   $G_1, G_2, G_3$ be graphs on the vertex set $[n+1]$ with edge sets
  \begin{align*} 
E(G_1)&=E(G)\backslash \{(a, i, -)\} \cup \{(a, i, +)\},  \\
E(G_2)&=E(G)\backslash \{(i, i, +)\} \cup \{(a, i, +)\},\tag{R6} \\ 
E(G_3)&=E(G)\backslash \{(a, i, -)\} \backslash \{(i, i, +)\} \cup \{(a, i, +)\}. 
\end{align*}

    We say that $G$ \textbf{reduces} to $G_1, G_2, G_3$ under the reduction rules (R1)-(R6). Figure \ref{fig:subdivrules} shows these reduction rules graphically.


 \begin{proposition} \label{red} Given a signed graph $G$  on the vertex set $[n+1]$, a vector $\aa \in \Z^{n+1}$, and two edges $e_1$ and $e_2$ of $G$ on which one of the reductions (R1)-(R6) can be performed yielding the graphs $G_1, G_2, G_3$, then 
 $$\F_G(\aa)=\F_{G_1}(\aa) \bigcup \F_{G_2}(\aa),\quad \F_{G_1}(\aa) \bigcap \F_{G_2}(\aa)=\F_{G_3}(\aa),\quad  \text{ and } \F_{G_1}(\aa)^\circ \bigcap \F_{G_2}(\aa)^\circ=\varnothing,$$
 
 \noindent where $\mathcal{P}^\circ$ denotes the interior of $\mathcal{P}$.
\end{proposition}

The proof of Proposition \ref{red} is left to the reader. Figure \ref{fig:subdivrules} and the definition of a flow polytope is all that is needed!

\begin{figure}
\centering
\includegraphics[height=8cm]{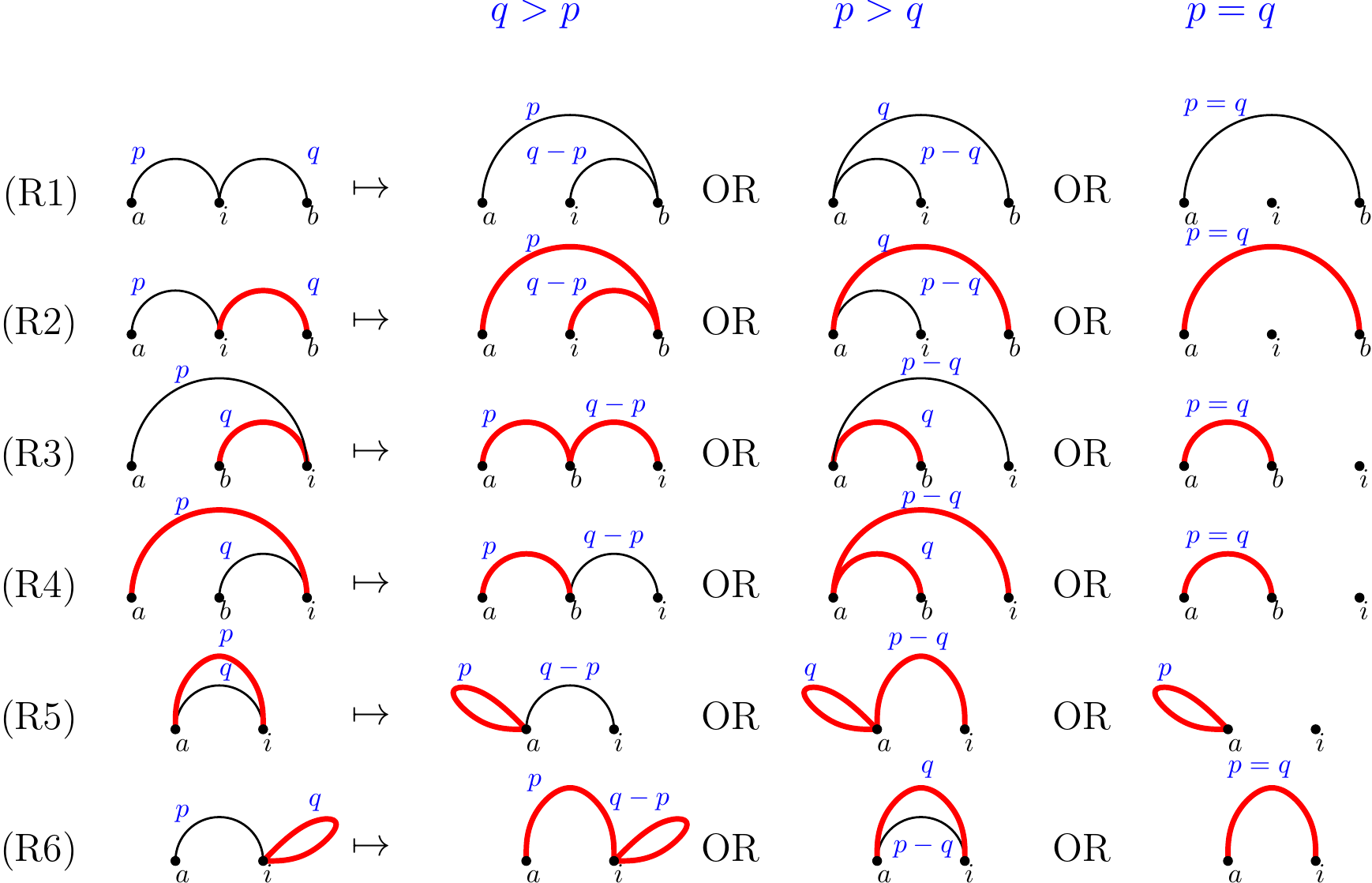}
\caption{Reduction rules from Equations (R1)-(R6). The original edges have flow $p$ and $q$. The outcomes have reassigned flows to preserve the original netflow on the vertices.}
 \label{fig:subdivrules}
\end{figure}

\section{Subdivision of  flow polytopes} \label{sec:flowssubdiv}

In this section we use the reduction rules for signed graphs given in Section \ref{sec:red}, following a specified order, to subdivide flow polytopes.  The main result of this section is the Subdivision Lemma as stated below, and again in Lemma \ref{lem:nct}. While the notation of this lemma seems complicated at first, the subsections below contain all the definitions and explanations necessary to understand it. This lemma is key in all our pursuits: it lies at the heart of the relationship between flow polytopes and Kostant partition functions. It also is a tool for systematic subdivisions, and as such calculating volumes of particular flow polytopes. 

\medskip

\noindent {\bf Subdivision Lemma.} {\it Let $G$ be a connected signed graph on the  vertex set  $[n+1]$ and $\mathcal{F}_G({\bf a})$ be its flow polytope for ${\bf a} \in \mathbb{Z}^{n+1}$. If for a fixed $i$ in $[n+1]$ $a_i=0$ and $G$ has no loops incident to vertex $i$, then the flow polytope subdivides as:
 \begin{equation}\label{lem:ncsub}
\mathcal{F}_G({\bf a}) = \bigcup_{T\in \mathcal{T}^{\pm}_{\I_i,\O_i}(\O_i^+)}
\mathcal{F}_{G^{(i)}_T}(a_1,\ldots,a_{i-1},\widehat{a_i},a_{i+1},\ldots,a_n, a_{n+1}),
\end{equation}
where $G_T^{(i)}$ are graphs on the vertex set $[n+1]\backslash \{i\}$  as defined in Section \ref{subsec:G_T}; and $\mathcal{T}^{\pm}_{\I_i,\O_i}(\O_i^+)$ is the set of signed trees as defined in Section \ref{nct}.}

First we define the trees, or equivalently compositions, that are important for the subdivision (Sections \ref{nct} and \ref{subsec:G_T}), then we define the order of application of reduction rules  and restate and prove the Subdivision Lemma (Section \ref{redorder}). In the next section we use this lemma to compute volumes of flow polytopes for both signless graphs $H$ and  signed graphs $G$.

\subsection{Noncrossing trees}\label{nct}

The subdivisions mentioned above are encoded by bipartite trees with negative and positive edges that are noncrossing. We start by defining such trees.

\begin{figure} 
\[
\begin{array}{lcr} 
\begin{tikzpicture}[scale=0.7] 
\node (l0) at (0,0.5) {};
\node (l1) at (0,0) {};
\node (l2) at (0,-0.5) {};
\node (l3) at (0,-1) {};
\node (r0) at (2,1) {};
\node (r1) at (2,0.5) {};
\node (r2) at (2,0) {};
\node (r3) at (2,-0.5) {};
\node (r4) at (2,-1) {};
\node (r5) at (2,-1.5) {};
\draw[thick] [-] (l0) to (r0);
\draw[thick] [-] (l1) to (r0);
\draw[thick] [-] (l1) to (r1);
\draw[thick] [-] (l1) to (r2);
\draw[thick] [-] (l2) to (r2);
\draw[thick] [-] (l2) to (r3);
\draw[thick] [-] (l3) to (r3);
\draw[thick] [-] (l3) to (r4);
\end{tikzpicture}
& 
\begin{tikzpicture}[scale=0.7] 
\node (l0) at (0,0.5) {};
\node (l1) at (0,0) {};
\node (l2) at (0,-0.5) {};
\node (l3) at (0,-1) {};
\node (r0) at (2,1) {};
\node (r1) at (2,0.5) {};
\node (r2) at (2,0) {};
\node (r3) at (2,-0.5) {};
\node (r4) at (2,-1) {};
\node (r5) at (2,-1.5) {};
\draw[very thick,red] [-] (l0) to (r0);
\draw[very thick,red] [-] (l1) to (r0);
\draw[thick] [-] (l1) to (r1);
\draw[thick] [-] (l1) to (r2);
\draw[thick] [-] (l2) to (r2);
\draw[thick] [-] (l2) to (r3);
\draw[thick] [-] (l3) to (r3);
\draw[red,very thick] [-] (l3) to (r4);
\end{tikzpicture}
&
\begin{tikzpicture}[scale=0.7] 
\node (l0) at (0,0.5) {};
\node (l1) at (0,0) {};
\node (l2) at (0,-0.5) {};
\node (l3) at (0,-1) {};
\node (r0) at (2,1) {};
\node (r1) at (2,0.5) {};
\node (r2) at (2,0) {};
\node (r3) at (2,-0.5) {};
\node (r4) at (2,-1) {};
\node (r5) at (2,-1.5) {};
\draw[red,very thick] [-] (l0) to (r0);
\draw[red,very thick] [-] (l1) to (r0);
\draw[thick] [-] (l1) to (r1);
\draw[very thick,red] [-] (l1) to (r2);
\draw[very thick,red] [-] (l2) to (r2);
\draw[thick] [-] (l2) to (r3);
\draw[thick] [-] (l3) to (r3);
\draw[red,very thick] [-] (l3) to (r4);
\end{tikzpicture}
\\
\text{(a)}\, T_1\in \mathcal{T}^-_{L,R}\phantom{xxxxxx} &\text{(b)}\, T_2\in \mathcal{T}^{\pm}_{L,R}((1,5))\phantom{xxxxxx} &\text{(c)}\, T_3 \in \mathcal{T}^{\pm}_{L,R}((1,3,5))
\end{array}
\]
\caption{Examples of bipartite noncrossing trees that are: (a) negative (composition $(1,0,1,1,0)$), (b) signed with $R^+=(1,5)$ (composition $(1^+,0^-,1^-,1^-,0^+)$), (c) signed with $R^+=(1,3,5)$ (composition $(1^+,0^-,1^+,1^-,0^+)$).}
\label{nctrees}
\end{figure}
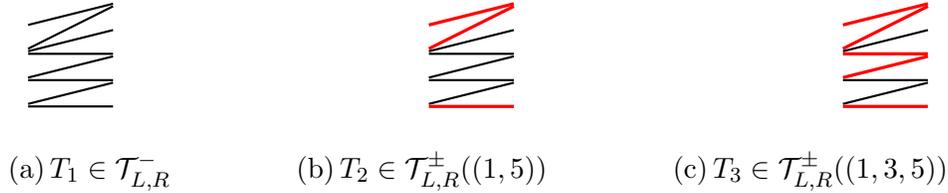

A {\bf negative bipartite noncrossing tree} $T$ with left vertices $x_1,\ldots,x_{\ell}$ and right vertices $x_{\ell+1},\ldots, x_{\ell+r}$ is a bipartite tree of negative edges that has no pair of edges $(x_p,x_{\ell+q},-), (x_t,x_{\ell+u},-)$ where $p<t$ and $q>u$. If $L$ and $R$ are the ordered sets $(x_1,\ldots,x_{\ell})$ and $(x_{\ell+1},\ldots,x_{\ell+r})$, let $\mathcal{T}^-_{L,R}$ be the set of such noncrossing bipartite trees. Note that $\# \mathcal{T}^-_{L,R}=\binom{\ell+r-2}{\ell-1}$, since they are in bijection with weak compositions of $\ell-1$ into $r$ parts. Namely, a tree $T$ corresponds to the composition of $(\text{indegrees $-1$})$ of the right vertices: $(b_1,\ldots,b_r)$, where $b_i$ denotes the number of edges incident to $x_{\ell+i}$ in $T$ minus $1$. See Figure \ref{nctrees}~(a) for an example of such a tree.

A {\bf signed bipartite noncrossing tree} is a bipartite noncrossing tree $T$ with negative $(\cdot,\cdot,-)$ and positive $(\cdot,\cdot,+)$ edges such that any right vertex is either incident to only negative edges or only positive edges. Let $\mathcal{T}^{\pm}_{L,R}(R^+)$ be the set of signed bipartite noncrossing trees with ordered left vertex set $L=(x_1,\ldots,x_{\ell})$, ordered right vertex set $R=(x_{\ell+1},\ldots,x_{\ell+r})$, and $R^+$ denoting the ordered set of right vertices  incident to only positive edges (the ordering of $R^+$ is inherited from the ordering of $R$). Note that for fixed $R^+$, $\# \mathcal{T}^{\pm}_{L,R}(R^+)=\#\mathcal{T}^-_{L,R}$, and we can encode such trees with a signed composition $(b_1^{\pm},b_2^{\pm},\ldots,b_r^{\pm})$ indicating whether the incoming edges to each right vertex are all positive or all negative and where $b_i$ denotes the number of edges incident to $x_{\ell+i}$ in $T$ minus $1$. See Figure \ref{nctrees}~(b)-(c) for examples of such trees. 

If either $L$ or $R$ is empty, the set $\mathcal{T}^{\pm}_{L,R} (R^+)$ consists of one element: the empty tree.

By abuse of notation we sometimes write $\mathcal{T}^{\pm}_{L,R}$, where $L$ and $R$ are sets as opposed to ordered sets. In these cases we assume that an order will be imposed on these sets. 




\subsection{Removing vertex $i$ from a signed graph $G$} \label{subsec:G_T} One of the points of the Subdivision Lemma is to start by a graph $G$ on the vertex set $[n+1]$ and to subdivide the flow polytope of $G$ into flow polytopes of graphs on a vertex set smaller than $[n+1]$. In this section we show the mechanics of this. We take a signed graph $G$ and replace incoming and outgoing edges of a fixed vertex $i$ by edges that avoid $i$ and come from a noncrossing tree $T$.
The outcome is a graph we denote by $G_T^{(i)}$ on the vertex set $[n+1]\backslash \{i\}$. To define this precisely we first introduce some notation. 

Given a signed graph $G$ and one of its vertices $i$, let $\mathcal{I}_i=\mathcal{I}_i(G)$ be the multiset of \textbf{incoming edges} to $i$, which are defined as negative edges of the form $(\cdot,i,-)$. Let $\mathcal{O}_i=\mathcal{O}_i(G)$ be the multiset of \textbf{outgoing edges} from $i$, which are defined as  edges of the form $(\cdot,i, +)$ and $(i,\cdot,\pm)$. Finally,  let $\mathcal{O}_i^{\pm}$ be the signed refinement of $\mathcal{O}_i$. Define $indeg_G(i):=\#\mathcal{I}_i(G)$ to be  the \textbf{indegree} of vertex $i$ in $G$. 

Assign an ordering to the sets $\mathcal{I}_i$ and $\mathcal{O}_i$ and consider a tree $T \in \mathcal{T}^{\pm}_{\mathcal{I}_i,\mathcal{O}_i}(\mathcal{O}_i^+)$. For each tree-edge $(e_1,e_2)$ of $T$ where $e_1=(r,i,-) \in \mathcal{I}_i$ and $e_2\in \mathcal{O}_i$ ($e_2=(i,s,\pm)$ or $(t,i,+)$), let $edge(e_1,e_2)$ be the following signed edge:
\begin{equation}\label{treeEdge}
edge(e_1,e_2)=\begin{cases}
(r,s,\pm)  & \text{ if } e_2=(i,s,\pm),\\
(r,t,+) & \text{ if } e_2=(t,i,+) \text{ and } r\leq t,\\
(t,r,+) & \text{ if } e_2=(t,i,+) \text{ and } r>t.
\end{cases}
\end{equation}
Note that if $e_1=(r,i,-)$ and $e_2=(r,i,+)$, then we allow $edge(e_1,e_2)$ to  be the loop $(r,r,+)$. Note also that $edge(e_1,e_2)$ is the edge corresponding to the type $C_{n+1}$ root $\vv(e_1)+\vv(e_2)$ where $\vv(e_1)$ and $\vv(e_2)$ are the positive type $C_{n+1}$ associated with $e_1$ and $e_2$. 

The graph $G^{(i)}_T$ is then defined as the graph obtained from $G$ by removing the vertex $i$ and all the edges of $G$ incident to $i$ and adding  the multiset of edges $\{\{edge(e_1,e_2) ~|~ (e_1,e_2)\in E(T)\}\}$. See Figure \ref{G_T} for examples of $G^{(i)}_T$.

\begin{remark} \label{compencoding}
If $T$ is given by a weak composition of $\#\mathcal{I}_i-1$ into $\#\mathcal{O}_i$ parts, say $(b_e)_{e\in \mathcal{O}_i}$, then:
\begin{compactitem}
\item[(i)] we record this composition by labeling the edges $e$ in $\mathcal{O}_i$ of $G$ with the corresponding part $b_e$. We can view this labeling as assigning a flow $b(e)=b_e$ to edges $e$ of $G$ in $\O_i$.  

\item[(ii)] The $b_e+1$ edges $(\cdot,e)$  in $T$ coming from the part $b_e$ of the composition will correspond to $b_e+1$ edges $edge(\cdot,e)$ in $G_T^{(i)}$. We think of these $b_e+1$ edges as one edge $e'$ coming from the original edge  $e$ in $G$, and $b_e$ {\bf truly new edges}.   
\end{compactitem}
\end{remark}

\begin{figure}
\centering
\subfigure[]{
\includegraphics[height=4.6cm]{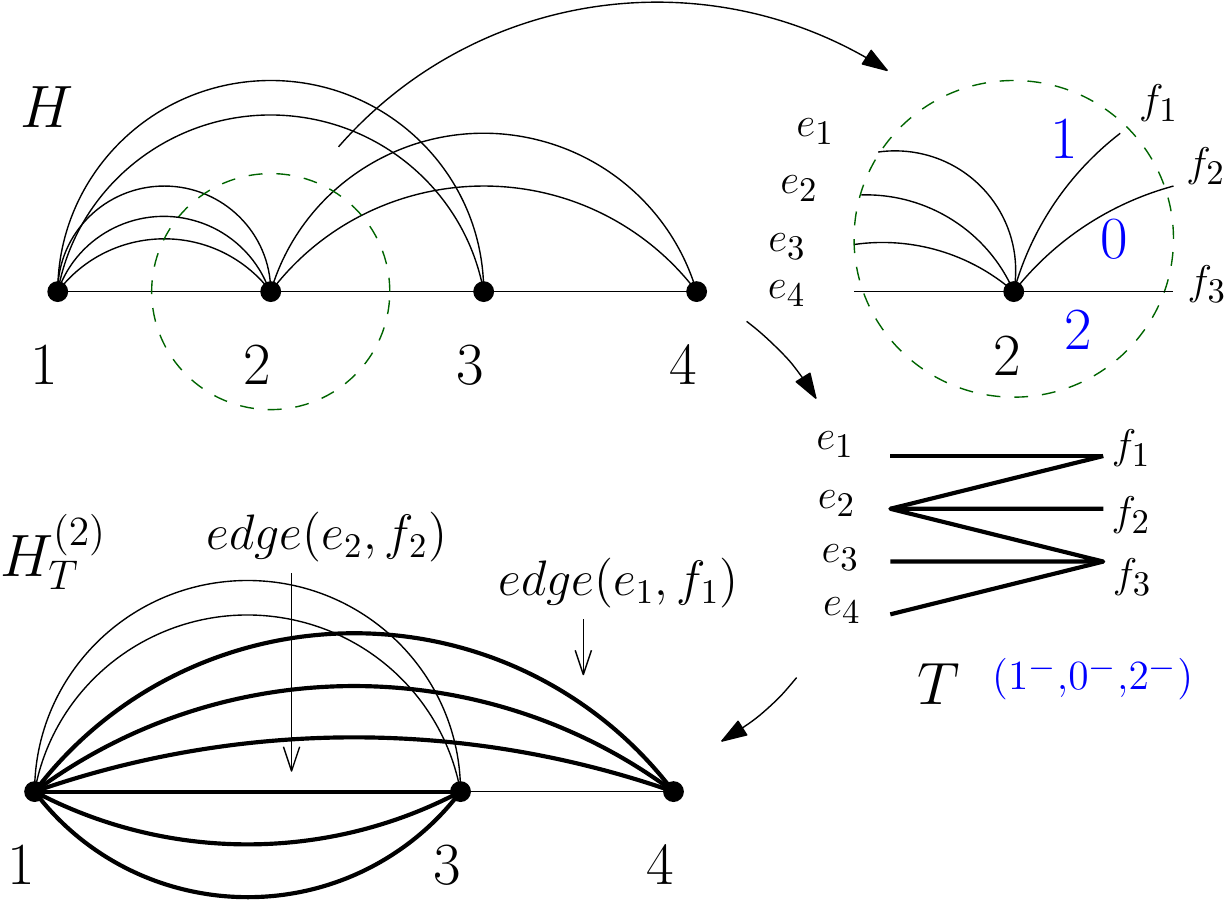}
\label{ttA}
}
\qquad
\subfigure[]{
\includegraphics[height=4.6cm]{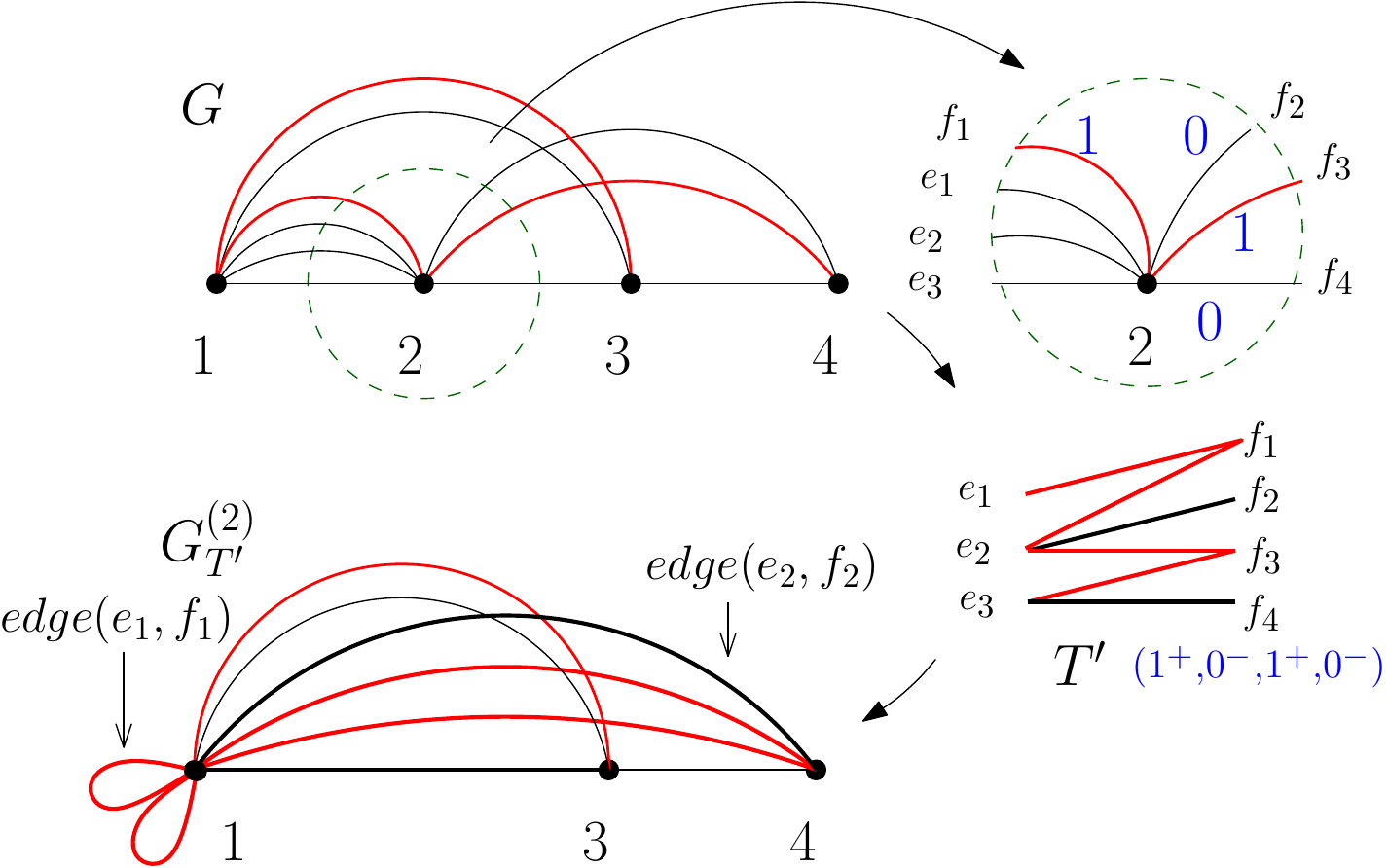}
\label{ttD}
}
\caption{Replacing the incident edges of vertex $2$ in (a) a graph $H$, with only negative edges, by a noncrossing tree $T$ encoded by the composition $(1^-,0^-,2^-)$ of $3=indeg_H(2)-1$. (b) a signed graph $G$ by a signed noncrossing tree $T'$ encoded by the composition $(1^+,0^-,1^+,0^-)$ of $2=indeg_G(2)-1$.} 
\label{G_T}
\end{figure}

The following is an easy consequence of the construction of $G_T^{(i)}$. (See the incoming and outgoing edges in $H$ and $H^{(2)}_{T}$; and $G$ and $G^{(2)}_{T'}$ in Figure~\ref{ttA},\ref{ttD}.)

\begin{proposition} \label{propG_T} Given a graph $G$ on the vertex set $[n+1]$, and $i, j \in [n+1]$, $i<j$, then the numbers of incoming and outgoing edges of vertex $j$ of the graph $G_T^{(i)}$ on the vertex set $[n+1]\backslash \{i\}$ built above are:
\begin{align} 
\#\mathcal{I}_j(G^{(i)}_T)&= 
\#\mathcal{I}_j(G) + \#\{\{\text{truly new edges } (k,j,-)~|~ k<i<j \}\}, \label{indeg} \\ 
\#\mathcal{O}_j(G^{(i)}_T) &= 
\#\mathcal{O}_j(G) + \#\{\{\text{truly new edges } (k,j,+) ~|~ k\neq i \}\}. \label{outdeg}
\end{align}
\end{proposition}

A statement of similar flavor as Proposition \ref{propG_T} can be made for $i>j$, but we omit it as we do not need it for our proofs.  

\begin{example}
For the graph $H$, with only negative edges, in Figure~\ref{ttA}: $\#\mathcal{I}_3(H^{(2)}_T)=\#\mathcal{I}_3(H)+2=5$, $\#\mathcal{I}_4(H^{(2)}_T)=\#\mathcal{I}_4(H)+1=4$ and $\#\mathcal{O}_3(H^{(2)}_T)=\#\mathcal{O}_3(H)=1$. 

For the signed graph $G$ in Figure~\ref{ttD}: $\#\mathcal{I}_3(G^{(2)}_{T'})=\#\mathcal{I}_3(G)+0=2$, $\#\mathcal{I}_4(G^{(2)}_{T'})=\#\mathcal{I}_4(G)+0=2$ and  $\#\mathcal{O}_3(G^{(2)}_{T'})=\#\mathcal{O}_3(G)=2$, and $\#\mathcal{O}_4(G^{(2)}_{T'})=\#\mathcal{O}_4(G)+1=2$.
\end{example}

Next, we give a subdivision of the flow polytope $\F_{G}$ of a signed graph $G$ in terms of flow polytopes $\F_{G_T^{(i)}}$ of graphs $G_T^{(i)}$. 

\subsection{Subdivision Lemma} \label{redorder}

In this subsection we are ready to state again the Subdivision Lemma, now with all the terminology defined, and prove it. We want to subdivide the flow polytope of a graph $G$ on the vertex set $[n+1]$. To do this we apply the reduction rules to incoming and outgoing edges of a vertex $i$ in $G$ with zero flow. Then by repeated application of reductions to this vertex, we can essentially delete this vertex from the resulting graphs, and as a result get to graphs on the vertex set $[n+1]\backslash \{i\}$. The Subdivision Lemma tells us exactly what these graphs, with a smaller vertex set, are. 

We have to specify in which order we do the reduction  at a given vertex $i$, since  at any given stage there might be several choices of pairs of edges to reduce. First we fix a linear order $\theta_\I$ on the multiset $\mathcal{I}_i(G)$ of incoming edges to vertex $i$, and a linear order $\theta_\O$ on the multiset $\mathcal{O}_i(G)$ of outgoing edges from vertex $i$. Recall that $\O_i(G)$ also includes edges $(a,i,+)$ where $a<i$. We choose the pair of edges to reduce in the following way: we pick the first available edge from $\I_i(G)$ and from $\O_i(G)$ according to the orders $\theta_{\I}$ and $\theta_{\O}$. At each step of the reduction, one outcome will have one fewer incoming edge and the other outcome will have one fewer outgoing edge. In each outcome, when we choose the next pair of edges to reduce we pick the next edge from $\I_i(G)$ and from $\O_i(G)$ that is still available. Since we only deal with the edges incident to vertex $i$,  for clarity we carry out the reductions on a graph $B$ representing these edges ordered by $\theta_{\I}$ and $\theta_\O$;  see Figure~\ref{ordervert}. The graph $B$ has left vertices $L$, a middle vertex $i$, and right vertices $R$; and edges $\{(e_t,i,-) \mid e_t\in L\}\cup \{(i,f_t,\pm) \mid f_t \in R\}$ where $\pm$ depends on the sign of $f_t$.


The Subdivision Lemma shows that when we follow this order to apply reductions to a vertex with zero flow the  outcomes are encoded by signed bipartite noncrossing trees.


\begin{figure}
\centering
\includegraphics[width=14cm]{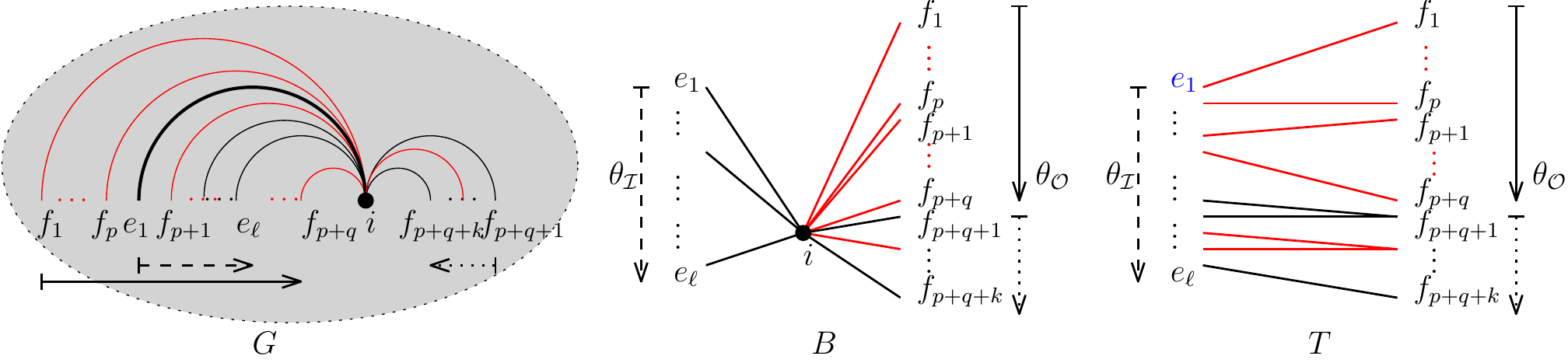}
\caption{Setting of Lemma~\ref{lem:nct} for edges incident to vertex $i$. We fix total orders $\theta_\I$ and $\theta_\O$ on $\I_i(G)$ (represented by dashed arrow) and $\O_i(G)$ (represented by straight and dotted arrows) respectively. The resulting bipartite trees are in $\mathcal{T}^{\pm}_{L,R}(R^+)$ where $L=\theta_\I(\I)$, $R=\theta_\O(\O)$ and $R^+=\theta_\O(\O^+)$.}
\label{ordervert}
\end{figure}

\begin{lemma}[Subdivision Lemma] \label{lem:nct}
Let $G$ be a connected signed graph on the vertex set $[n+1]$ and $\mathcal{F}_G({\bf a})$ be its flow polytope for ${\bf a} \in \mathbb{Z}^{n+1}$. Assume that for a fixed $i$ in $[n+1]$,  $a_i=0$ and $G$ has no loops incident to vertex $i$. Fix linear orders $\theta_\I$ and $\theta_\O$ on $\I_i(G)$ and $\O_i(G)$ respectively. If we apply the reduction rules to edges incident to vertex $i$ following the linear orders, then the flow polytope subdivides as:
 \begin{equation}\label{lem:ncsub}
\mathcal{F}_G({\bf a}) = \bigcup_{T\in \mathcal{T}^{\pm}_{L,R}(R^+)} \mathcal{F}_{G^{(i)}_T}(a_1,\ldots,a_{i-1},\widehat{a_i},a_{i+1},\ldots,a_{n+1}),
\end{equation}
where $G_T^{(i)}$ is as defined in Section \ref{subsec:G_T}; and $\mathcal{T}^{\pm}_{L,R}(R^+)$ is the set of signed trees with $L=\theta_\I(\I_i), R=\theta_\O(\O_i)$ and $R^+=\theta_\O(\O^+_i)$. 
\end{lemma}

\begin{proof} If  either $\#\I_i(G)$ or  $\#\O_i(G)$ equals $0$, then the edges incident to $i$ are all forced to have flow $0$ and there is nothing to prove. In this case we call the graph obtained by deleting vertex $i$ and  the edges incident to it from $G$ the {\em final outcome} of the reduction process on the graph $G$. If $\#\I_i(G)$ and  $\#\O_i(G)$ are at least $1$, 
apply the reduction rules (see Figure~\ref{fig:subdivrules} for the rules) to pairs of edges incident to $i$ following the orders $\theta_\I$ and $\theta_\O$.  Each step of the reduction takes a graph $G$ and gives two graphs $G_1$ and $G_2$ as defined by the reduction rules (R1)-(R6) given in Section \ref{sec:red}. Note that $G_1$ and $G_2$ differ from $G$ in exactly two edges: we deleted one edge incident to $i$ from $G$ and added an edge not incident to $i$;  the new edge will not take part in any other reduction on vertex $i$. We continue the reduction until we obtain a graph $\widetilde{G}$ with $\#\I_i(\widetilde{G})$ or $\#\O_i(\widetilde{G})$ equaling $1$.  Note that once we obtain such a  graph $\widetilde{G}$, then one edge incident to vertex $i$ will have a forced value for its flow, and we can replace  the graph $\widetilde{G}$ on the vertex set $[n+1]$ with a graph $\widehat{G}$ on the vertex set $[n+1]\backslash \{i\}$, whose flow polytope is equivalent to that of  $\widetilde{G}$. 
  We also call such graphs $\widehat{G}$ the {\em final outcomes} of the reduction process on the graph $G$.  See Figure \ref{final-outcomes} for an illustration of how to get from $\widetilde{G}$  to $\widehat{G}$. A graph $\widehat{G}$ is considered a {\bf final outcome} for $G$ with respect to the orders $\theta_\I$ and $\theta_\O$ if it is obtained in one of the two ways described above. 


 We show by induction on $c_G:=\#\mathcal{I}_i(G) + \#\mathcal{O}_i(G)$ that the  final outcomes of the reduction process on the graph $G$ with respect to the orders $\theta_\I$ and $\theta_\O$ as described above  are exactly the graphs $G_T^{(i)}$ for all noncrossing bipartite trees $T$ in $\mathcal{T}^{\pm}_{L,R}(R^+)$ where $L=\theta_\I(\I_i), R=\theta_\O(\O_i)$ and $R^+=\theta_\O(\O^+_i)$. Recall that such trees are in bijection with signed compositions $(b_e)_{e\in R}$ of $\#\I_i(G)-1$ into $\#\O_i(G)$ parts.

The base case, when $c_G=1$, is trivial by the above discussion.
Consider a graph $G$ with $c_G>1$. If either  $\#\I_i(G)$ or $\#\O_i(G)$ equals $0$ or $1$, then doing as described in the first paragraph of the proof  we are done. If both $\#\I_i(G)$ and $\#\O_i(G)$ are greater than  $1$, then
 using linear orders $\theta_\I$ and $\theta_\O$ we pick the next available pair of edges to reduce. The pair will be an incoming negative edge $e_1=(a,i,-)$ and an outgoing edge $f_1=(i,b,\pm)$ or $(b,i,+)$.
We do one of the reduction presented in Figure~\ref{fig:subdivrules} and obtain graphs $G'$ and $G''$ with a new edge $edge(e_1,f_1)=(a,b,\pm)$ and without $(i,b,\pm)$ or $(a,i,-)$ respectively (see Figure~\ref{inductivestep}). For both $G'$ and $G''$ we have $c_{G'}=c_{G''}=c_G-1$ ($\O_i(G')=\O_i(G)\backslash \{(i,b,\pm)\}$ and $\I_i(G'')=\I_i(G)\backslash \{(a,i,-)\}$). By induction, the final outcomes of the reduction on $G'$ are $G'^{(i)}_{T'}$ where $T'$ are noncrossing bipartite trees in $\mathcal{T}^{\pm}_{L,R\backslash \{f_1\}}$. But $T'\cup (a,b,\pm)$ is  still a noncrossing bipartite tree (since we follow the orders $\theta_{\I_i(G)}$ and $\theta_{\O_i(G)}$), $G'^{(i)}_{T'} = G^{(i)}_{T' \cup (a,b,\pm)}$ and the set $\{T'\cup (a,b,\pm) \mid T' \in \mathcal{T}^{\pm}_{L,R\backslash \{f_1\}}\}$ is exactly the set of trees in $\mathcal{T}^{\pm}_{L,R}(R^+)$ with $b_{f_1} = indeg(f_1)-1=0$ (see Figure~\ref{inductivestep}). Let $\mathcal{T}^{(b_{f_1}=0)}$ be the set of such trees. Similarly, by induction, the final outcomes of the reduction on $G''$ are the graphs $G^{(i)}_T$ for all trees $T$ in $\mathcal{T}^{\pm}_{L,R}(R^+)$ where $b_{f_1}=indeg(f_1)-1>0$. Let $\mathcal{T}^{(b_{f_1}>0)}$ be the set of such trees. Since $\mathcal{T}^{\pm}_{L,R}(R^+)=\mathcal{T}^{(b_{f_1}=0)} \cup \mathcal{T}^{(b_{f_1}>0)}$ where the union is disjoint, then from $G$ we obtain the final outcomes $G^{(i)}_T$ where $T\in \mathcal{T}^{\pm}_{L,R}(R^+)$.

\begin{figure}
\subfigure[]{
\includegraphics[width=4.5cm]{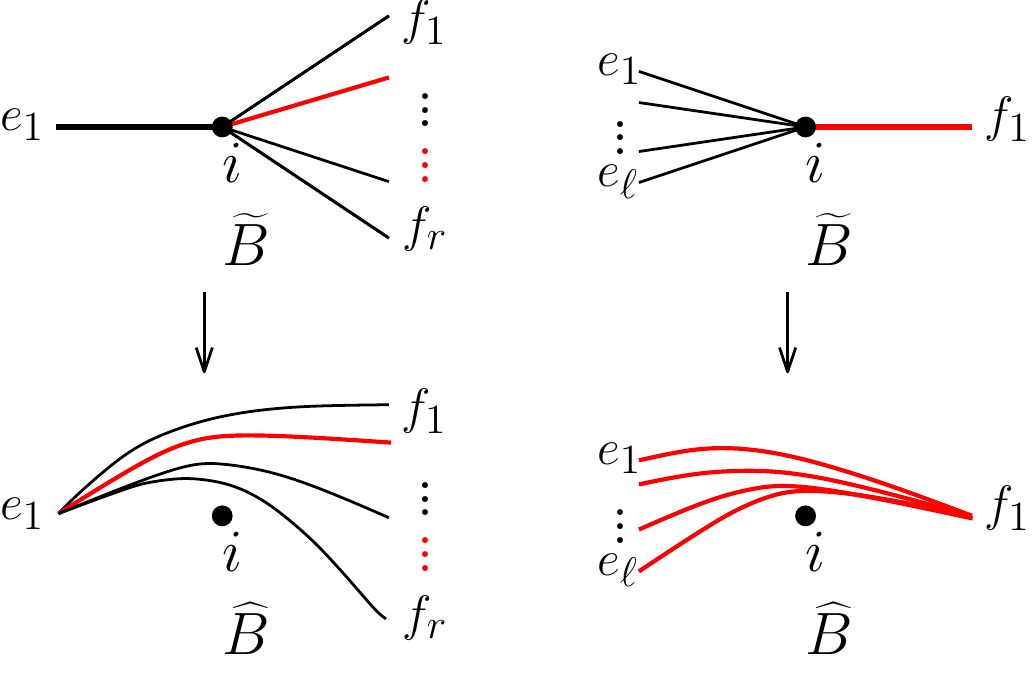}
\label{final-outcomes}
}
\subfigure[]{
\includegraphics[width=4.5cm]{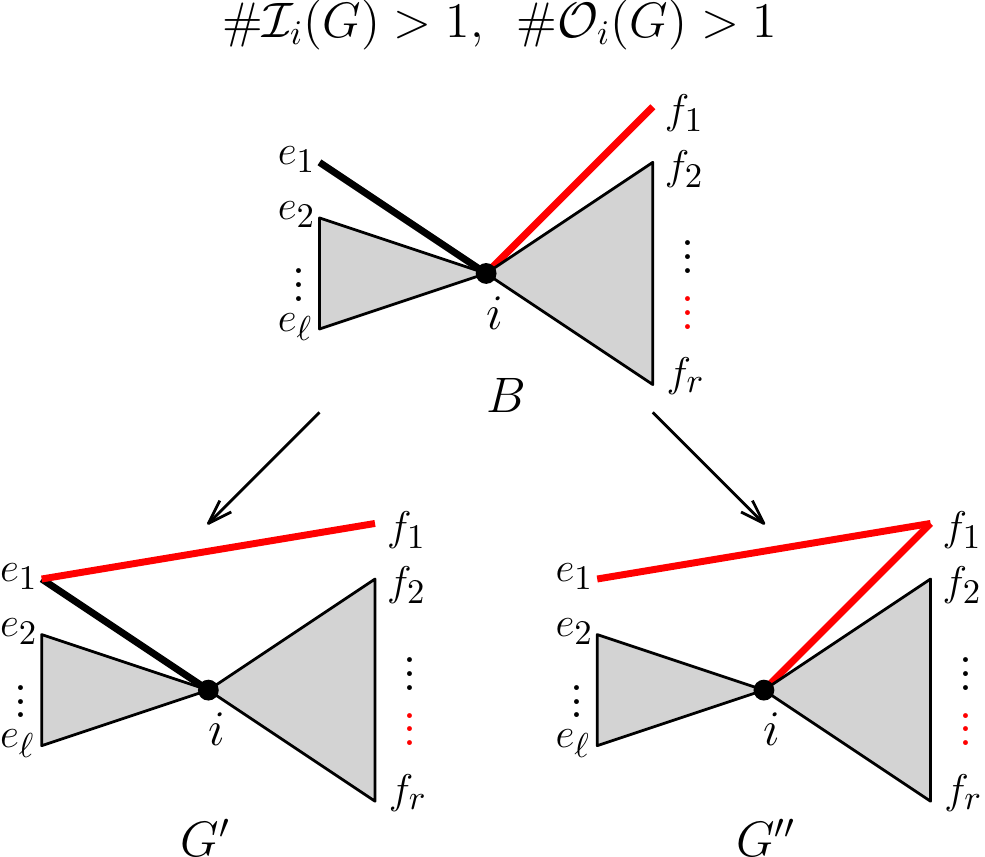}
\label{inductivestep}
}
\subfigure[]{
\includegraphics[width=6.5cm]{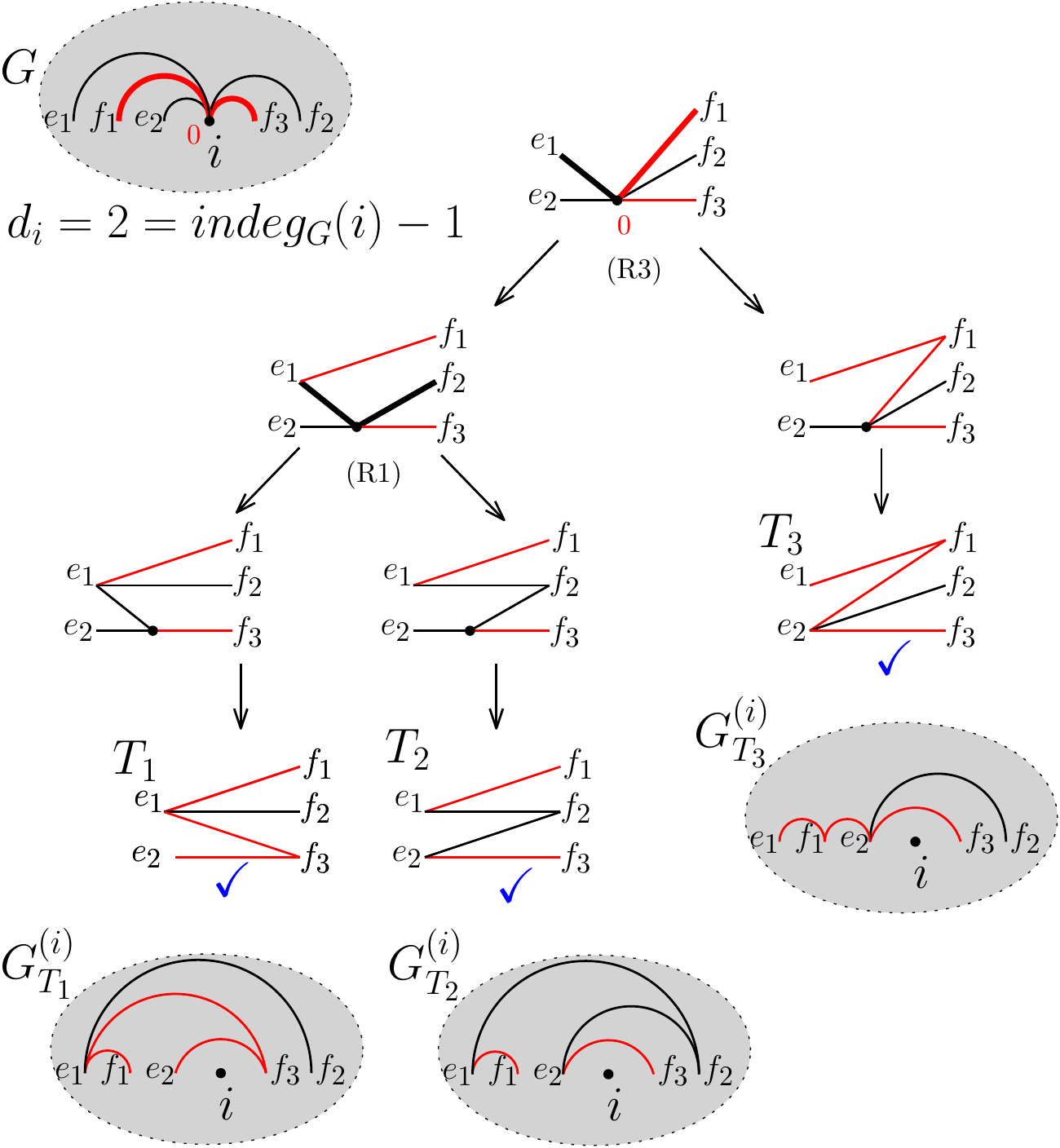}
\label{fig:exSubDivLemma}
}
\caption{(a) Obtaining final outcomes $\widehat{G}$ when either $\#\mathcal{I}_i(\widetilde{G})=1$ or $\#\mathcal{O}_i(\widetilde{G})=1$. (b) Inductive step in proof of the Subdivision Lemma. (c) Example of a subdivision (the selected edges to reduce are {\bf bo\textcolor{red}{ld}}). The final outcomes indicated by $\textcolor{blue}{\checkmark}$ are indexed by signed trees in $\mathcal{T}^{\pm}_{(e_1,e_2),(f_1,f_2,f_3)}(f_1,f_3)$ or equivalently the compositions $(0^+,0^-,1^+), (0^+,1^-,0^+)$, and $(1^+,0^-,0^+)$.}
\end{figure}

So from the reduction we get flow polytopes $\F_{G^{(i)}_T}(\g')$ where $\g'=(a_1,\ldots,a_{i-1},\widehat{a_i},a_{i+1},\ldots,a_{n+1})$ and $T$ is in $\mathcal{T}^{\pm}_{L,R}(R^+)$. Thus, by repeated application of Proposition~\ref{red}, it will follow that $\F_G(\g)$ subdivides as a union of $\F_{G^{(i)}_T}(\g')$ for all trees $T$ in $\mathcal{T}^{\pm}_{L,R}(R^+)$ as desired. 
\end{proof}


See Figure \ref{fig:exSubDivLemma} for an example of a subdivision into final outcomes that are indexed by noncrossing bipartite trees.

In the next section, we apply Lemma \ref{lem:nct} to compute the volume of the flow polytope  $\F_G({\bf a})$ where $G$ is a signed graph and ${\bf a}=(2,0,\ldots,0)$, the highest root of the root system $C_{n+1}$. As a motivation and to highlight differences, we first use a special case of the Subdivision Lemma, as done by Postnikov and Stanley \cite{p,S}, to compute the volume of the polytope $\F_{H}(1,0,\ldots,0,-1)$ where $H$ is a graph with only negative edges.

\section{Volume of flow polytopes} \label{sec:vol}

In this section we use the Subdivision Lemma (Lemma \ref{lem:nct}) on flow polytopes $\F_H(1,0,\ldots,0,-1)$, where $H$ is a graph with only negative edges, and on $\F_G(2,0,\ldots,0)$, where $G$ is a signed graph, to prove the formulae for their volume given in Theorem \ref{sp} (\cite{p,S}) and Theorem \ref{mm}, respectively. (Recall from Equation~\eqref{eq:Minkowski}, that every flow polytope $\F_{G'}(\aa)$ is a Minkowski sum of such flow polytopes.) To establish the connection between the volume of flow polytopes and Kostant partition functions for signed graphs, in Section~\ref{volsD} we introduce the notion of dynamic Kostant partition functions, which specializes to  Kostant partition functions in the case of graphs with only negative edges. 

\subsection{A correspondence between integer flows and simplices in a triangulation  of $\F_H(\ee_1-\ee_{n+1})$, where $H$  only has negative edges} \label{vols}

Let $H$ be a connected graph on the  vertex set $[n+1]$ and {\em only} negative edges, and $\F_H(1,0,\ldots,0,-1)$ be its flow polytope where $(1,0,\ldots,0,-1)\in \mathbb{Z}^{n+1}$. We apply Lemma \ref{lem:nct} successively to vertices $2,3,\ldots,n$ which have zero netflow. At the end we obtain the subdivision:
\begin{equation} \label{typeA:subdiv}
\F_H(1,0,\ldots,0,-1) = \bigcup_{T^-_{n}}\cdots  \bigcup_{T^-_3} \bigcup_{T^-_2} \F_{((\cdots  (H_{T^-_2}^{(2)})^{(3)}_{T^-_3}\cdots)_{T^-_{n}}^{(n)}}(1,-1),
\end{equation}
where $T_i^-$ are noncrossing trees with only negative edges. See Figure~\ref{tA} for an example of an outcome of a subdivision of an instance of $\F_H(1,0,\ldots,0,-1)$. The graph $H_n:=((\cdots  (H_{T^-_2}^{(2)})^{(3)}_{T^-_3}\cdots)_{T^-_{n}}^{(n)}$  consists of two vertices, $1$ and $n+1$ and $\#E(H)-n+1$ edges between them. Thus $\F_{H_n}(1,-1)$ is an $(\#E(H)-n)$-dimensional simplex with normalized unit volume (see Example~\ref{mainexs} (i)). Therefore, $\vol(\F_{H_n}(1,0,\ldots,0,-1))$ is the number of choices of bipartite noncrossing trees $T^-_{2},\ldots,T^-_n$ where $T^-_{i+1}$ encodes a composition of $\#\mathcal{I}_{i+1}(H_i)-1$ with $\#\mathcal{O}_{i+1}(H_i)$ parts. The next result by Postnikov and Stanley \cite{p,S} shows that this number of tuples of trees is also the number of certain integer flows on $H$. We reproduce their proof to motivate and highlight the differences with the case of signed graphs discussed in the next subsection. This result also appeared in \cite[Prop. 34]{BV1}.

\begin{theorem}[\cite{p,S}] \label{sp}  Given a loopless (signless) connected graph $H$ on the vertex set $[n+1]$, let $d_i=indeg_H(i)-1$ for $i \in \{2, \ldots, n\}$. Then, the normalized volume $\vol(\F_H(1,0,\ldots,0,-1))$ of the flow polytope associated to graph $H$ is 
\[
\vol(\F_H(1,0,\ldots,0,-1))=K_H(0,d_2, \ldots, d_n, -\sum_{i=2}^n d_i),
\]
where $K_H$ is the Kostant partition function of $H$.
\end{theorem}

\begin{example}[Application of Theorem \ref{sp}]
The flow polytope $\F_H(1,0,0,-1)$ for the negative graph $H$ in Figure \ref{figexthm1} (a) has normalized volume $4$. This is the number of flows on $H$ with netflow $(0,d_2,d_3,d_4)=(0,3,2,-5)$ where $d_i=indeg_H(i)-1$, i.e., $K_{H}(0,3,2,-5)=4$. The four flows are in Figure \ref{figexthm1} (b).
\end{example}

\noindent \textit{Proof of Theorem \ref{sp}.}
For this proof, let $H_i:=(\cdots  (H_{T^-_2}^{(2)})^{(3)}_{T^-_3}\cdots)_{T^-_{i}}^{(i)}$ for $i=2,\ldots,n$. From Equation \eqref{typeA:subdiv} and the discussion immediately after, we have that $\vol(\F_H(1,0,\ldots,0,-1))$ is the number of choices of noncrossing bipartite trees $(T^-_{2},\ldots,T^-_n)$ where $T_{i+1}^-$ encodes a composition of $\#\mathcal{I}_{i+1}(H_i)-1$ with $\#\mathcal{O}_{i+1}(H_i)$ parts. We give a correspondence between $\left(H;(T^-_{2},\ldots,T^-_n)\right)$ and integer ${\bf a}$-flows on $H$ where ${\bf a}=(0,d_2,\ldots,d_n,-\sum_{i=2}^nd_i)$.  The proof is then complete since by  Remark~\ref{flowsKG} these integer flows are counted by $K_H(0,d_2,\ldots,d_n,-\sum_{i=2}^n d_i)$. 

To give the correspondence between $\left(H;(T^-_{2},\ldots,T^-_n)\right)$ and integer ${\bf a}$-flows on $H$ where ${\bf a}=(0,d_2,\ldots,d_n,-\sum_{i=2}^nd_i)$, recall that the tree $T^-_{i+1}$ is given by a composition $(b^{(i+1)}_e)_{e\in \mathcal{O}_{i+1}(H_i)}$ of $\#\mathcal{I}_{i+1}(H_i)-1$ into $\#\mathcal{O}_{i+1}(H_i)$ parts. By Remark \ref{compencoding} (i), we can encode this composition by assigning a flow $b(e)=b^{(i+1)}_e$ to edges $e$ of $H_i$ in $\mathcal{O}_{i+1}(H_i)$. But since $H$ and $H_i$ consist only of negative edges, iterating Proposition \ref{propG_T} we see that 
\begin{equation}\label{outdegA}
\#\mathcal{O}_{i+1}(H_i)=\#\mathcal{O}_{i+1}(H).
\end{equation}
(In fact these sets are equal). Therefore, we can also encode the compositions {\em on the edges of $H$}. So, for $i=2,\ldots,n$ we record compositions $(b_e^{(i)})_{e\in \mathcal{O}_i(H)}$ (and thus the trees $T^-_{i}$) as flows $b(e)=b_e^{(i)}$ on $e\in\mathcal{O}_i(H)$ of $H$. For $i=1$, we assign flows $b(e)=0$ for $e\in \mathcal{O}_1(H)$. See the third column of Figure~\ref{tA} for an example of this encoding. Next we calculate the netflow on vertex $i+1$ of $H$:
\begin{align}
\sum_{e\in \mathcal{O}_{i+1}(H)}b(e)&=\#I_{i+1}(H_i)-1, \label{eq:Aoutdeg} \\ 
\sum_{e\in \mathcal{I}_{i+1}(H)}b(e)&=\#\{\{\text{truly new edges } (\cdot,i+1,-)\}\}. \label{eq:Aindeg}
\end{align}
Where Equation \eqref{eq:Aoutdeg} follows since $(b_e^{(i+1)})_{e\in \mathcal{O}_{i+1}(H)}$ is a composition of $\#I_{i+1}(H_i)-1$. Equation \eqref{eq:Aindeg} follows from Remark \ref{compencoding} (ii). Then using these two equations the netflow $a_{i+1}$ of vertex $i+1$ is
\begin{align*}
a_{i+1} &= \sum_{e \in \mathcal{O}_{i+1}(H)}b(e) - \sum_{e\in \mathcal{I}_{i+1}(H)} b(e)\\
&=\left(\#\mathcal{I}_{i+1}(H_i)-1\right)-\#\{\{\text{truly new edges }(\cdot,i+1,-)\}\}.\\
\intertext{By Proposition~\ref{propG_T} we get $\#\mathcal{I}_{i+1}(H_i)=\#\mathcal{I}_{i+1}(H)+\#\{\{\text{truly new edges }(\cdot,i+1,-)\}\}$, so}
a_{i+1} &= \left(\#\mathcal{I}_{i+1}(H)+\#\{\{\text{truly new edges }(\cdot,i+1,-)\}\}-1\right)- \#\{\{\text{truly new edges }(\cdot,i+1,-)\}\}.
\end{align*}  
So $a_{i+1}= \#\mathcal{I}_{i+1}(H)-1=indeg_H(i+1)-1 = d_{i+1}$. Thus we have a map from $\left(H; (T_2^-,\ldots,T^-_n)\right)$ to an integer ${\bf a}$-flow in $H$ where ${\bf a}=(0,d_2,\ldots,d_n,-\sum_{i=2}^n d_i)$. See Figure~\ref{tA} for an example of this map.

Next we show this map is bijective by building its inverse. Given such an integer flow on $H$, we read off the flows on the edges of $\mathcal{O}_i(H)$ for  $i=2,\ldots,n$ in clockwise order and obtain a weak composition of $N_i:=\sum_{e\in \mathcal{O}_i(H)} b(e)$ with $\#\mathcal{O}_i(H)$ parts. Next, we encode each of these compositions as noncrossing trees $T^-_i$. We know that $\#\mathcal{O}_{i+1}(H)=\#\mathcal{O}_{i+1}(H_i)$ and it is not hard to show by induction on $i$ that $N_{i+1}=\#\mathcal{I}_{i+1}(H_i)-1$ where $H_{i}=(\cdots(H_{T_2^-}^{(2)})^{(3)}_{T_3^-}\cdots)^{(i)}_{T^-_{i}}$. Thus $T_i^-$ encodes a composition of $\#\I_{i+1}(H_i)-1$ with $\#\O_{i+1}(H_i)$ parts. Therefore, we also have a map from an integer $\g$-flow in $H$ to a tuple $(H;(T_2^-,\ldots,T_n^-))$.


It is easy to see that the two maps described above are inverses of each other. This shows the first of these maps is the correspondence we desired.
\qed

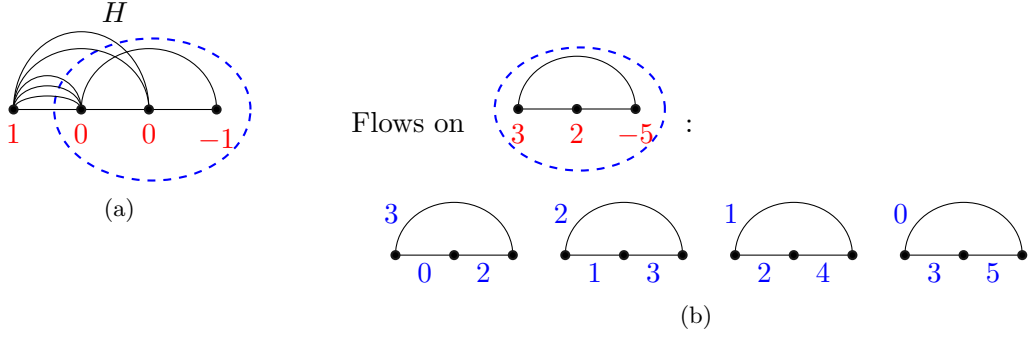
\begin{figure}
\subfigure[]{
\begin{tikzpicture}[scale=0.45] 
\draw[blue,thick,dashed] [-] (5,0) arc (0:360:2.9cm and 2.1cm);
\draw [-] (0,0) arc (0:180:1cm and 0.7cm); 
\draw [-] (0,0) arc (0:180:1cm and 0.4cm);
\draw [-] (-2,0) arc (180:0:2cm and 1.8cm);
\draw [-] (-2,0) arc (180:0:2cm and 2.3cm); 
\draw [-] (0,0) arc (0:180:1cm);
\draw [-] (4,0) to (-2,0);
\draw [-] (0,0) arc (180:0:2cm and 1.8cm);
\node [style=bb,label=below:\textcolor{red}{$1$}] (1) at (-2, 0) {};
\node [style=bb,label=below:\textcolor{red}{$0$}] (2) at (0, 0) {};
\node [style=bb,label=below:\textcolor{red}{$0$}] (3) at (2, 0) {};
\node [style=bb] (4) at (4, 0) {};
\node [label=below:\textcolor{red}{$-1$}] at (4,0) {};
\node [label=above:$H$] (00) at (1,2) {};
\end{tikzpicture}
}
\qquad
\subfigure[]{
\begin{tabular}{l}
\raisebox{20pt}{Flows on }
\raisebox{5pt}{
\begin{tikzpicture}[scale=0.39]
\draw[blue,thick,dashed] [-] (5,0) arc (0:360:2.9cm and 2.1cm);
\node at (0,1.7) {};
\draw[white,thick,dashed] [-] (3,0) arc (0:360:1cm); 
\draw [-] (4,0) to (0,0);
\draw [-] (0,0) arc (180:0:2cm and 1.8cm);
\node [style=bb,label=below:\textcolor{red}{$3$}] (2) at (0, 0) {};
\node [style=bb,label=below:\textcolor{red}{$2$}] (3) at (2, 0) {};
\node [style=bb,label=below:\textcolor{red}{$-5$}] (4) at (4, 0) {};
\end{tikzpicture}
}
\raisebox{20pt}{:}\\
{
\begin{tabular}{llll}
\begin{tikzpicture}[scale=0.39]
\node at (0,1.7) {};
\draw[white,thick,dashed] [-] (3,0) arc (0:360:1cm); 
\draw [-] (4,0) to (0,0);
\draw [-] (0,0) arc (180:0:2cm and 1.8cm);
\node [style=bb] (2) at (0, 0) {};
\node [style=bb] (3) at (2, 0) {};
\node [style=bb] (4) at (4, 0) {};
\node [label=left:$\textcolor{blue}{3}$] (b) at (0.8, 1.4) {};
\node [label=below:$\textcolor{blue}{0}$] (c) at (1, 0.4) {};
\node [label=below:$\textcolor{blue}{2}$] (b) at (3, 0.4) {};
\end{tikzpicture}
&
\begin{tikzpicture}[scale=0.39]
\node at (0,1.7) {};
\draw[white,thick,dashed] [-] (3,0) arc (0:360:1cm); 
\draw [-] (4,0) to (0,0);
\draw [-] (0,0) arc (180:0:2cm and 1.8cm);
\node [style=bb] (2) at (0, 0) {};
\node [style=bb] (3) at (2, 0) {};
\node [style=bb] (4) at (4, 0) {};
\node [label=left:$\textcolor{blue}{2}$] (b) at (0.8, 1.4) {};
\node [label=below:$\textcolor{blue}{1}$] (c) at (1, 0.4) {};
\node [label=below:$\textcolor{blue}{3}$] (b) at (3, 0.4) {};
\end{tikzpicture}
&
\begin{tikzpicture}[scale=0.39]
\node at (0,1.7) {};
\draw[white,thick,dashed] [-] (3,0) arc (0:360:1cm); 
\draw [-] (4,0) to (0,0);
\draw [-] (0,0) arc (180:0:2cm and 1.8cm);
\node [style=bb] (2) at (0, 0) {};
\node [style=bb] (3) at (2, 0) {};
\node [style=bb] (4) at (4, 0) {};
\node [label=left:$\textcolor{blue}{1}$] (b) at (0.8, 1.4) {};
\node [label=below:$\textcolor{blue}{2}$] (c) at (1, 0.4) {};
\node [label=below:$\textcolor{blue}{4}$] (b) at (3, 0.4) {};
\end{tikzpicture}
&
\begin{tikzpicture}[scale=0.39]
\node at (0,1.7) {};
\draw[white,thick,dashed] [-] (3,0) arc (0:360:1cm); 
\draw [-] (4,0) to (0,0);
\draw [-] (0,0) arc (180:0:2cm and 1.8cm);
\node [style=bb] (2) at (0, 0) {};
\node [style=bb] (3) at (2, 0) {};
\node [style=bb] (4) at (4, 0) {};
\node [label=left:$\textcolor{blue}{0}$] (b) at (0.8, 1.4) {};
\node [label=below:$\textcolor{blue}{3}$] (c) at (1, 0.4) {};
\node [label=below:$\textcolor{blue}{5}$] (b) at (3, 0.4) {};
\end{tikzpicture}
\end{tabular}
}
\end{tabular}
}
\caption{Example of Theorem \ref{sp} to find $\vol \F_H(1,0,0,-1)=K_H(0,3,2,-5)=4$: (a) Graph $H$ with negative edges, (b) the four flows on $H$ with netflow $(0,d_2,d_3,d_4)=(0,3,2,-5)$ where $d_i=indeg_H(i)-1$.}
\label{figexthm1}
\end{figure}

We now look at computing the normalized volume of $\F_G({\bf a})$ where $G$ is a signed graph and ${\bf a}=(2,0,\ldots,0)$. 

\subsection{A correspondence between dynamic integer flows and simplices in a triangulation  of $\F_G(2\ee_1)$, where $G$ is a signed graph} \label{volsD}


Let $G$ be a connected signed graph on the  vertex set $[n+1]$ and ${\bf a}=(2,0,\ldots,0)$. In order to subdivide the polytope $\F_{G}({\bf a})$, we follow the same first steps as in the previous case. Mainly:

We apply Lemma \ref{lem:nct} successively to vertices $2,3,\ldots,n+1$ which have zero netflow. At the end we obtain:
\begin{equation} \label{typeD:subdiv}
\F_G(2,0,\ldots,0) = \bigcup_{T_{n+1}}\cdots \bigcup_{T_3} \bigcup_{T_2} \F_{(\cdots  (G_{T_2}^{(2)})_{T_3}^{(3)}\cdots)_{T_{n+1}}^{(n+1)}}(2).
\end{equation}
See Figure \ref{tD} for an example of an outcome of a subdivision of an instance of $\F_G(2,0,\ldots,0)$. In this case, $G_{n+1}:=(\cdots  (G_{T_2}^{(2)})_{T_3}^{(3)}\cdots)_{T_{n+1}}^{(n+1)}$ is a graph consisting of one vertex with $\#E(G)-n$ positive loops.  Thus, $\F_{G_{n+1}}(2)$ is an $(\#E(G)-n-1)$-dimensional simplex with normalized unit volume (see Example~\ref{mainexs} (ii)). Therefore, $\vol(\F_G(2,0,\ldots,0))$ is the number of choices of signed noncrossing bipartite trees $T_{2}, T_{3},\ldots,T_{n+1}$ where $T_{i+1}$ encodes a composition of $\#\mathcal{I}_{i+1}(G_i)-1$ with $\#\mathcal{O}_{i+1}(G_i)$ parts. However, instead of a correspondence between $(G; (T_2,T_3,\ldots,T_{n+1}))$ and the usual integer flows on $G$, there is a correspondence with a special kind of integer flow on $G$ that we call {\bf dynamic integer flow}.

Next, we motivate the need of these new integer flows. Let $G_i:=(\cdots  (G_{T_2}^{(2)})\cdots)_{T_{i}}^{(i)}$ be as defined above. The tree $T_{i+1}$ is given by a signed composition $(b_e^{(i+1)})_{e\in \mathcal{O}_{i+1}(G_i)}$ of $\#\mathcal{I}_{i+1}(G_i)-1$ into $\# \mathcal{O}_{i+1}(G_i)$ parts. And again, by Remark \ref{compencoding} (i), we can encode the composition by assigning a flow $b(e)=b_e^{(i+1)}$ to edges $e$ of $G_i$ in $\mathcal{O}_{i+1}(G_i)$. However, contrary to Equation \eqref{outdegA}, iterating Proposition \ref{propG_T} we get
\begin{equation}\label{outdegD}
\#\mathcal{O}_{i+1}(G_i) \geq \#\mathcal{O}_{i+1}(G),
\end{equation}
e.g., in Figure \ref{G_T} (b), $\mathcal{O}_4(G)=\{\{(2,4,+)\}\}$ and $\mathcal{O}_4(G^{(2)}_{T_2})=\{\{(1,4,+),(1,4,+)\}\}$ where one of these two edges is a truly new positive edge (see Remark~\ref{compencoding}). Thus, we cannot encode the compositions as flows on a  {\em fixed} graph $G$ but rather on a graph $G$ {\em and} $\#\mathcal{O}_{i+1}(G_i)-\#\mathcal{O}_{i+1}(G)$ additional positive edges incident to $i$. The number of such additional edges is determined by the integer flows assigned to previous positive edges $(k,i+1,+)$ of $G$ where $k<i+1$. This is what we mean by dynamic flow since the graph $G$ changes as the flow is assigned. The next definition makes this precise.

\begin{figure}
\subfigure[]{
\includegraphics[width=7.55cm]{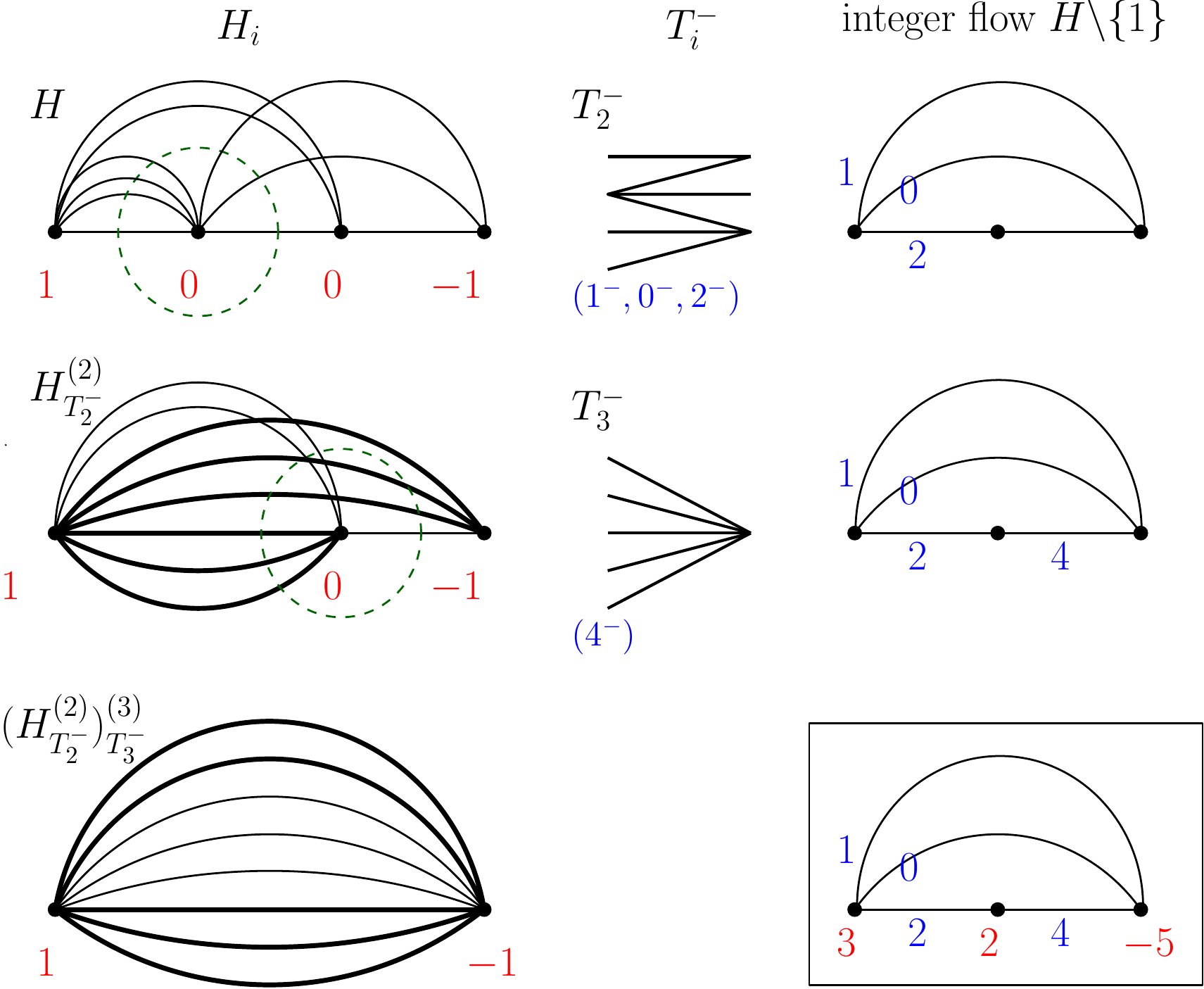}
\label{tA}
}
\qquad
\subfigure[]{
\includegraphics[width=7.55cm]{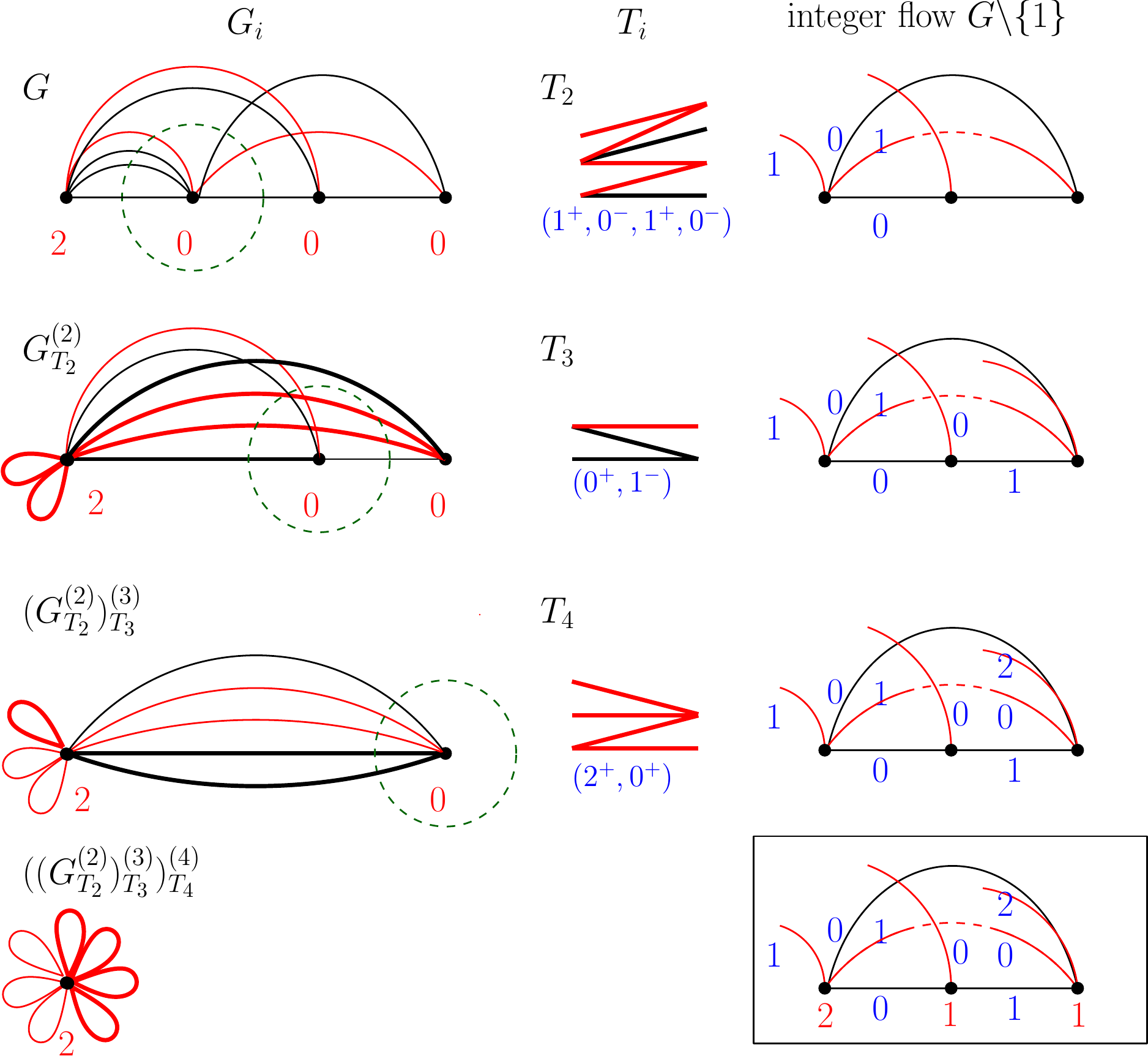}
\label{tD}
}
\caption{
Example of the subdivision to find the volume of (a) $\F_H(1,0,0,-1)$ for $H$ with only negative edges and of (b) $\F_G(2,0,0,0)$ for signed $G$. The subdivision is encoded by noncrossing trees $T_{i+1}$ that are equivalent to compositions $\textcolor{blue}{(b_1,\ldots,b_r)}$ of $\#\mathcal{I}_{i+1}(H_i)-1$ ($\#\mathcal{I}_{i+1}(G_i)-1$ resp.) with $\#\mathcal{O}_{i+1}(H_i)$ ($\#\mathcal{O}_{i+1}(G_i)$ resp.) parts. These trees or compositions are recorded by the integer (dynamic) flow on $H\backslash \{1\}$ ($G\backslash\{1\}$ resp.) in the box with netflow $(d_2,d_3,-d_2-d_3)=(3,2,-5)$ where $d_i=indeg_i(H)-1$ ($(d_2,d_3,d_4)=(2,1,1)$ where $d_i=indeg_i(G)-1$ resp.).}
\label{fig:subdivhrA}
\end{figure}

\begin{definition}[Dynamic integer flow] \label{def:dynflows}
Given a signed graph $G$ and an edge $e=(i,j,+)$ of $G$, we will regard $e=(i,j,+)$ as two positive {\em half-edges} $(i,\varnothing,+)$ and $(\varnothing,j,+)$ that still have ``memory'' of being together (see Figure \ref{fig:dynflow} (a)). We assign nonnegative {\em integer} flows $b_{\ell}(e)$ and $b_{r}(e)$ to the left and right halves of the positive edge, starting at the left half-edge. Once we assign $b_{\ell}(e)$ units of flow, {\em we add $b_{\ell}(e)$ {\bf extra right positive half-edges}}  {\em incident to $j$}. Any right positive half-edge $e'$ is assigned a nonnegative integer flow $b_r(e')$ (whether it was an extra right positive half-edge, or an original one). When we assign a nonnegative integer flow to a right positive half-edge no edges of any kind are added making the process of adding extra edges to the graph finite.

An analogue of Equation \eqref{eqn:flow} still holds:
\begin{equation}\label{eqn:dynflow}
 \sum_{e \in \I_i(G)} b(e)+a_i= \sum_{e \in \O_i^-(G)} b(e) + \sum_{e=(i,\cdot,+) \in \O_i^+(G)} b_{\ell}(e)+\sum_{e=(\cdot,i,+)\in \O_i^+(G)} b_{r}(e) + \sum_{\text{extra right} \atop \text{half-edges $e'=(\varnothing,i,+)$}} b_{r}(e'),
\end{equation}
where $a_i$ is the netflow at vertex $i$ and $\I_i(G)$, $\O^-_i(G)$, and $\O^+_i(G)$ are the incoming and outgoing to edges  as defined in Section~\ref{subsec:G_T}. We call these integer ${\bf a}$-flows {\bf dynamic}.
\end{definition}


\begin{example}
For the signed graph $G$ in Figure \ref{fig:dynflow} (a) with only one positive edge $e=(1,3,+)$, we give three of its $17$ integer dynamic flows with netflow $(2,1,1)$ where we add $b_{\ell}(e)=0,1$ and $2$ right half-edges respectively.
\end{example}

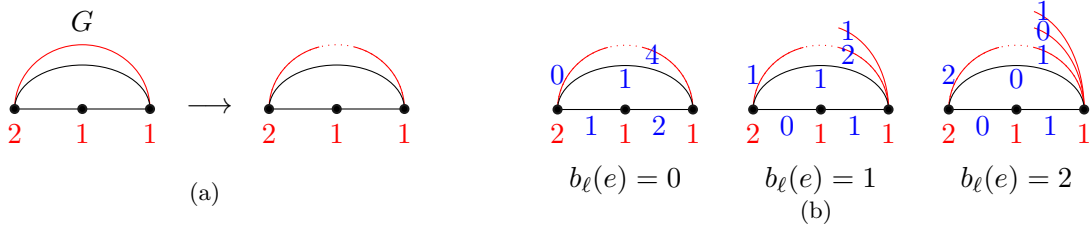
\begin{figure}
\centering
\subfigure[]{
\raisebox{5pt}{
\begin{tikzpicture}[scale=0.45] 
\draw [-] (4,0) to (0,0);
\draw[red] [-] (0,0) arc (180:0:2cm and 1.9cm);
\draw [-] (0,0) arc (180:0:2cm and 1.3cm);
\node [style=bb,label=below:$\textcolor{red}{2}$] (2) at (0, 0) {};
\node [style=bb,label=below:$\textcolor{red}{1}$] (3) at (2, 0) {};
\node [style=bb,label=below:$\textcolor{red}{1}$] (4) at (4, 0) {};
\node [label=below:$G$] (a) at (2, 3.5) {};
\end{tikzpicture}
}
\raisebox{20pt}{$\longrightarrow$}
\raisebox{5pt}{
\begin{tikzpicture}[scale=0.45] 
\draw [-] (4,0) to (0,0);
\draw[red,dotted] [-] (0,0) arc (180:0:2cm and 1.9cm);
\draw[red] [-] (0,0) arc (180:105:2cm and 1.90cm);
\draw[red] [-] (4,0) arc (0:75:2cm and 1.90cm);
\draw [-] (0,0) arc (180:0:2cm and 1.3cm);
\node [style=bb,label=below:$\textcolor{red}{2}$] (2) at (0, 0) {};
\node [style=bb,label=below:$\textcolor{red}{1}$] (3) at (2, 0) {};
\node [style=bb,label=below:$\textcolor{red}{1}$] (4) at (4, 0) {};
\end{tikzpicture}
}
}
\qquad
\subfigure[]{
\raisebox{28pt}{
\begin{tabular}{ccc}
\begin{tikzpicture}[scale=0.45] 
\draw [-] (4,0) to (0,0);
\draw[red,dotted] [-] (0,0) arc (180:0:2cm and 1.9cm);
\draw[red] [-] (0,0) arc (180:105:2cm and 1.90cm);
\draw[red] [-] (4,0) arc (0:75:2cm and 1.90cm);
\draw [-] (0,0) arc (180:0:2cm and 1.3cm);
\node [style=bb,label=below:$\textcolor{red}{2}$] (2) at (0, 0) {};
\node [style=bb,label=below:$\textcolor{red}{1}$] (3) at (2, 0) {};
\node [style=bb,label=below:$\textcolor{red}{1}$] (4) at (4, 0) {};
\node [label=below:$\textcolor{blue}{0}$] (b) at (0, 1.9) {};
\node [label=below:$\textcolor{blue}{1}$] (c) at (2, 1.8) {};
\node [label=below:$\textcolor{blue}{1}$] (c) at (1, 0.4) {};
\node [label=below:$\textcolor{blue}{2}$] (c) at (3, 0.4) {};
\node [label=below:$\textcolor{blue}{4}$] (c) at (2.8, 2.5) {};
\end{tikzpicture}
&
\begin{tikzpicture}[scale=0.45] 
\draw [-] (4,0) to (0,0);
\draw[red,dotted] [-] (0,0) arc (180:0:2cm and 1.9cm);
\draw[red] [-] (0,0) arc (180:105:2cm and 1.90cm);
\draw[red] [-] (4,0) arc (0:75:2cm and 1.90cm);
\draw[red] [-] (4,0) arc (0:75:2cm and 2.5cm);
\draw [-] (0,0) arc (180:0:2cm and 1.3cm);
\node [style=bb,label=below:$\textcolor{red}{2}$] (2) at (0, 0) {};
\node [style=bb,label=below:$\textcolor{red}{1}$] (3) at (2, 0) {};
\node [style=bb,label=below:$\textcolor{red}{1}$] (4) at (4, 0) {};
\node [label=below:$\textcolor{blue}{1}$] (b) at (0, 1.9) {};
\node [label=below:$\textcolor{blue}{1}$] (c) at (2, 1.8) {};
\node [label=below:$\textcolor{blue}{0}$] (c) at (1, 0.4) {};
\node [label=below:$\textcolor{blue}{1}$] (c) at (3, 0.4) {};
\node [label=below:$\textcolor{blue}{2}$] (c) at (2.8, 2.5) {};
\node [label=below:$\textcolor{blue}{1}$] (c) at (2.8, 3.2) {};
\end{tikzpicture}
&
\begin{tikzpicture}[scale=0.45] 
\draw [-] (4,0) to (0,0);
\draw[red,dotted] [-] (0,0) arc (180:0:2cm and 1.9cm);
\draw[red] [-] (0,0) arc (180:105:2cm and 1.90cm);
\draw[red] [-] (4,0) arc (0:75:2cm and 1.90cm);
\draw[red] [-] (4,0) arc (0:75:2cm and 2.5cm);
\draw[red] [-] (4,0) arc (0:75:2cm and 3cm);
\draw [-] (0,0) arc (180:0:2cm and 1.3cm);
\node [style=bb,label=below:$\textcolor{red}{2}$] (2) at (0, 0) {};
\node [style=bb,label=below:$\textcolor{red}{1}$] (3) at (2, 0) {};
\node [style=bb,label=below:$\textcolor{red}{1}$] (4) at (4, 0) {};
\node [label=below:$\textcolor{blue}{2}$] (b) at (0, 1.9) {};
\node [label=below:$\textcolor{blue}{0}$] (c) at (2, 1.8) {};
\node [label=below:$\textcolor{blue}{0}$] (c) at (1, 0.4) {};
\node [label=below:$\textcolor{blue}{1}$] (c) at (3, 0.4) {};
\node [label=below:$\textcolor{blue}{1}$] (c) at (2.8, 2.5) {};
\node [label=below:$\textcolor{blue}{0}$] (c) at (2.8, 3.2) {};
\node [label=below:$\textcolor{blue}{1}$] (c) at (2.8, 3.8) {};
\end{tikzpicture}
\\
$b_{\ell}(e)=0$ & $b_{\ell}(e)=1$ & $b_{\ell}(e)=2$
\end{tabular}
}
}
\caption{Example of dynamic flow: (a) signed graph $G$ with positive edge $e$ split into two half-edges, (b) three of the $17$ dynamic integer flows where $b_{\ell}(e)=0,1$, and $2$ so that zero, one and two right positive half-edges are added respectively.}
\label{fig:dynflow}
\end{figure}

We translate~\eqref{outdeg} from Proposition~\ref{propG_T} in terms of integer dynamic flows to turn~\eqref{outdegD} into an equality.

\begin{lemma} \label{lemG_Tdynflow} Let $G$ and $G_T^{(i)}$ be as in Proposition~\ref{propG_T}, where the tree $T$ is given by a weak composition $(b_e)_{e\in \O_i(G)}$ of $\#\I_i(G)-1$. Assign an integer dynamic flow $b(e)=b_e$ to the edges $e$ in $\O_i(G)$; more precisely, assign the flow $b(e)=b_e$ to the negative edges $e$ in  $\O_i(G)$ and assign the flow $b_{\ell}(e)=b_e$ to left part of the positive edges $e$ in  $\O_i(G)$. Finally,  add extra positive right half-edges according the definition of dynamic flow. Then for $j$ in $[n+1]$, $j>i$ we have that
\begin{align}
\#\mathcal{O}_j(G^{(i)}_T) &= \#\mathcal{O}_j(G) + \#\{\{\text{extra right half-edges $(\varnothing,j,+)$}\}\}.
\end{align}
\end{lemma}

\begin{example}
For the signed graph $G$  and the tree $T_2$ encoding the composition 
\[
(b_{(1,2,+)},b_{(2,4,-)},b_{(2,4,+)},b_{(2,3,-)})=(1^+,0^-,1^+,0^-)
\] 
in Figure~\ref{ttD}, $\#\O_4(G^{(2)}_{T_2})=2$, $\#\O_4(G)=1$ and there is one extra right half-edge $(\varnothing,4,+)$ added when we assign the flow $b_{\ell}(2,\varnothing,+)=1$ to the left half-edge of $(2,4,+)$.
\end{example}

We also introduce an analogue of the Kostant partition function that counts integer dynamic flows of signed graphs. Later, in Section~\ref{genseriesKdyn} we will give a generating series for this function.

\begin{definition}[Dynamic Kostant partition function] \label{defdkpf}
Given a signed graph $G$ on the vertex set $[n+1]$ and ${\bf a}$ a vector in $\mathbb{Z}^{n+1}$, the {\bf dynamic Kostant partition function} $K_G^{\text{dyn}}({\bf a})$ is the number of integer dynamic ${\bf a}$-flows in $G$.  
\end{definition}

\begin{example}
For the signed graph $G$ in Figure \ref{fig:dynflow}, $K_G^{\text{dyn}}(2,1,1)=17$.
\end{example}


We are now ready to state and prove our main result as an application of the technique we developed. 

\begin{theorem} \label{mm}
Given a loopless connected signed graph $G$ on the vertex set $[n+1]$, let $d_i=indeg_G(i)-1$ for $i\in \{2,\ldots,n\}$. The normalized volume $\vol(\F_G)$ of the flow polytope associated to graph $G$ is   
\[
\vol(\F_G(2,0,\ldots,0)) = K_G^{\text{dyn}}(0,d_2, \ldots, d_n,d_{n+1}).
\]
\end{theorem}  

\begin{example}[Application of  Theorem \ref{mm}]
The flow polytope $\F_G(2,0,0,0)$ for the signed graph $G$ in Figure \ref{figexthm2} (a) has normalized volume $5$. This is the number of dynamic integer flows on $G$ with netflow $(0,d_2,d_3,d_4)=(0,1,0,1)$ where $d_i=indeg_G(i)-1$. The five dynamic integer flows are in Figure \ref{figexthm2} (b).

\begin{figure}
\centering
\subfigure[]{
\begin{tikzpicture}[scale=0.45] 
\draw[blue,thick,dashed] [-] (4.5,0) arc (0:360:2.5cm and 2.1cm);
\draw [-] (0,0) arc (0:180:1cm and 0.7cm); 
\draw[red] [-] (-2,0) arc (180:0:2cm and 2.3cm); 
\draw [-] (4,0) to (-2,0);
\draw [-] (0,0) arc (180:0:2cm and 1.8cm);
\draw[red] [-] (0,0) arc (180:0:2cm and 1.3cm);  
\node (0) at (-3.4,0) {};
\node [style=bb,label=below:$\textcolor{red}{2}$] (1) at (-2, 0) {};
\node [style=bb,label=below:$\textcolor{red}{0}$] (2) at (0, 0) {};
\node [style=bb,label=below:$\textcolor{red}{0}$] (3) at (2, 0) {};
\node [style=bb,label=below:$\textcolor{red}{0}$] (4) at (4, 0) {};
\node [label=above:$G$] (00) at (1,2) {};
\end{tikzpicture}
}
\subfigure[]{
\begin{tabular}{l}
\raisebox{20pt}{Dynamic flows on }
\raisebox{5pt}{
\begin{tikzpicture}[scale=0.45] 
\draw[blue,thick,dashed] [-] (4.5,0) arc (0:360:2.5cm and 2.1cm);
\draw[white,thick,dashed] [-] (3,0) arc (0:360:1cm); 
\draw[red] [-] (2,0) arc (0:30:2cm); 
\draw [-] (4,0) to (0,0);
\draw [-] (0,0) arc (180:0:2cm and 1.8cm);
\draw[red,dotted] [-] (0,0) arc (180:0:2cm and 1.30cm);
\draw[red] [-] (0,0) arc (180:100:2cm and 1.3cm);
\draw[red] [-] (4,0) arc (0:80:2cm and 1.3cm);  
\node [style=bb,label=below:$\textcolor{red}{1}$] (2) at (0, 0) {};
\node [style=bb,label=below:$\textcolor{red}{0}$] (3) at (2, 0) {};
\node [style=bb,label=below:$\textcolor{red}{1}$] (4) at (4, 0) {};
\end{tikzpicture}
}
\raisebox{20pt}{:}\\
\begin{tabular}{lllll}
\begin{tikzpicture}[scale=0.45] 
\draw[white,thick,dashed] [-] (3,0) arc (0:360:1cm); 
\draw[red] [-] (2,0) arc (0:90:0.7cm); 
\draw [-] (4,0) to (0,0);
\draw [-] (0,0) arc (180:0:2cm and 1.8cm);
\draw[red] [-] (0,0) arc (180:100:2cm and 1.3cm);
\draw[red] [-] (4,0) arc (0:80:2cm and 1.3cm);
\draw[red,dotted] [-] (0,0) arc (180:0:2cm and 1.30cm);
\node [style=bb] (2) at (0, 0) {};
\node [style=bb] (3) at (2, 0) {};
\node [style=bb] (4) at (4, 0) {};
\node [label=below:$\textcolor{blue}{1}$] (b) at (2, 3) {};
\node [label=below:$\textcolor{blue}{0}$] (c) at (0.5, 1.6) {};
\node [label=below:$\textcolor{blue}{0}$] (c) at (1, 0.3) {};
\node [label=below:$\textcolor{blue}{0}$] (a) at (2.1, 1.5) {};
\node [label=below:$\textcolor{blue}{0}$] (b) at (3, 0.3) {};
\node [label=below:$\textcolor{blue}{0}$] (a) at (3, 1.8) {};
\end{tikzpicture}
&
\begin{tikzpicture}[scale=0.45] 
\draw[white,thick,dashed] [-] (3,0) arc (0:360:1cm); 
\draw[red] [-] (2,0) arc (0:90:0.7cm); 
\draw [-] (4,0) to (0,0);
\draw [-] (0,0) arc (180:0:2cm and 1.8cm);
\draw[red] [-] (0,0) arc (180:100:2cm and 1.3cm);
\draw[red] [-] (4,0) arc (0:80:2cm and 1.3cm);
\draw[red,dotted] [-] (0,0) arc (180:0:2cm and 1.30cm);
\node [style=bb] (2) at (0, 0) {};
\node [style=bb] (3) at (2, 0) {};
\node [style=bb] (4) at (4, 0) {};
\node [label=below:$\textcolor{blue}{0}$] (b) at (2, 3) {};
\node [label=below:$\textcolor{blue}{0}$] (c) at (0.5, 1.6) {};
\node [label=below:$\textcolor{blue}{1}$] (c) at (1, 0.3) {};
\node [label=below:$\textcolor{blue}{0}$] (a) at (2.1, 1.5) {};
\node [label=below:$\textcolor{blue}{1}$] (b) at (3, 0.3) {};
\node [label=below:$\textcolor{blue}{2}$] (a) at (3, 1.8) {};
\end{tikzpicture}
&
\begin{tikzpicture}[scale=0.45] 
\draw[white,thick,dashed] [-] (3,0) arc (0:360:1cm); 
\draw[red] [-] (2,0) arc (0:90:0.7cm); 
\draw [-] (4,0) to (0,0);
\draw [-] (0,0) arc (180:0:2cm and 1.8cm);
\draw[red] [-] (0,0) arc (180:100:2cm and 1.3cm);
\draw[red] [-] (4,0) arc (0:80:2cm and 1.3cm);
\draw[red,dotted] [-] (0,0) arc (180:0:2cm and 1.30cm);
\node [style=bb] (2) at (0, 0) {};
\node [style=bb] (3) at (2, 0) {};
\node [style=bb] (4) at (4, 0) {};
\node [label=below:$\textcolor{blue}{0}$] (b) at (2, 3) {};
\node [label=below:$\textcolor{blue}{0}$] (c) at (0.5, 1.6) {};
\node [label=below:$\textcolor{blue}{1}$] (c) at (1, 0.3) {};
\node [label=below:$\textcolor{blue}{1}$] (a) at (2.1, 1.5) {};
\node [label=below:$\textcolor{blue}{0}$] (b) at (3, 0.3) {};
\node [label=below:$\textcolor{blue}{1}$] (a) at (3, 1.8) {};
\end{tikzpicture}
&
\begin{tikzpicture}[scale=0.45] 
\draw[white,thick,dashed] [-] (3,0) arc (0:360:1cm); 
\draw[red] [-] (2,0) arc (0:90:0.7cm); 
\draw [-] (4,0) to (0,0);
\draw [-] (0,0) arc (180:0:2cm and 1.8cm);
\draw[red] [-] (0,0) arc (180:100:2cm and 1.3cm);
\draw[red] [-] (4,0) arc (0:80:2cm and 1.3cm);
\draw[red,dotted] [-] (0,0) arc (180:0:2cm and 1.30cm);
\draw[red] [-] (4,0) arc (0:80:2cm and 1cm);    
\node [style=bb] (2) at (0, 0) {};
\node [style=bb] (3) at (2, 0) {};
\node [style=bb] (4) at (4, 0) {};
\node [label=below:$\textcolor{blue}{0}$] (b) at (2, 3) {};
\node [label=below:$\textcolor{DarkGreen}{\bf 1}$] (c) at (0.5, 1.6) {};
\node [label=below:$\textcolor{blue}{0}$] (c) at (1, 0.3) {};
\node [label=below:$\textcolor{blue}{0}$] (a) at (2.1, 1.5) {};
\node [label=below:$\textcolor{blue}{0}$] (b) at (3, 0.3) {};
\node [label=below:$\textcolor{blue}{1}$] (a) at (2.7, 2.2) {};
\node [label=below:$\textcolor{blue}{0}$] (b) at (3.3, 1.5) {};
\end{tikzpicture}
&
\begin{tikzpicture}[scale=0.45] 
\draw[white,thick,dashed] [-] (3,0) arc (0:360:1cm); 
\draw[red] [-] (2,0) arc (0:90:0.7cm); 
\draw [-] (4,0) to (0,0);
\draw [-] (0,0) arc (180:0:2cm and 1.8cm);
\draw[red] [-] (0,0) arc (180:100:2cm and 1.3cm);
\draw[red] [-] (4,0) arc (0:80:2cm and 1.3cm);
\draw[red,dotted] [-] (0,0) arc (180:0:2cm and 1.30cm);
\draw[red] [-] (4,0) arc (0:80:2cm and 1cm);    
\node [style=bb] (2) at (0, 0) {};
\node [style=bb] (3) at (2, 0) {};
\node [style=bb] (4) at (4, 0) {};
\node [label=below:$\textcolor{blue}{0}$] (b) at (2, 3) {};
\node [label=below:$\textcolor{DarkGreen}{\bf 1}$] (c) at (0.5, 1.6) {};
\node [label=below:$\textcolor{blue}{0}$] (c) at (1, 0.3) {};
\node [label=below:$\textcolor{blue}{0}$] (a) at (2.1, 1.5) {};
\node [label=below:$\textcolor{blue}{0}$] (b) at (3, 0.3) {};
\node [label=below:$\textcolor{blue}{0}$] (a) at (2.7, 2.2) {};
\node [label=below:$\textcolor{blue}{1}$] (b) at (3.3, 1.5) {};
\end{tikzpicture}
\end{tabular}
\end{tabular}
}
\caption{Example of Theorem \ref{mm} to find $\vol \F_G(2,0,0,0)=K_G^{\text{dyn}}(0,1,0,1)=5$: (a) Signed graph $G$, (b) the five dynamic flows on $G$ with netflow $(0,d_2,d_3,d_4)=(0,1,0,1)$ where $d_i=indeg_G(i)-1$ (the last two flows have an additional right positive half-edge).} 
\label{figexthm2}
\end{figure}
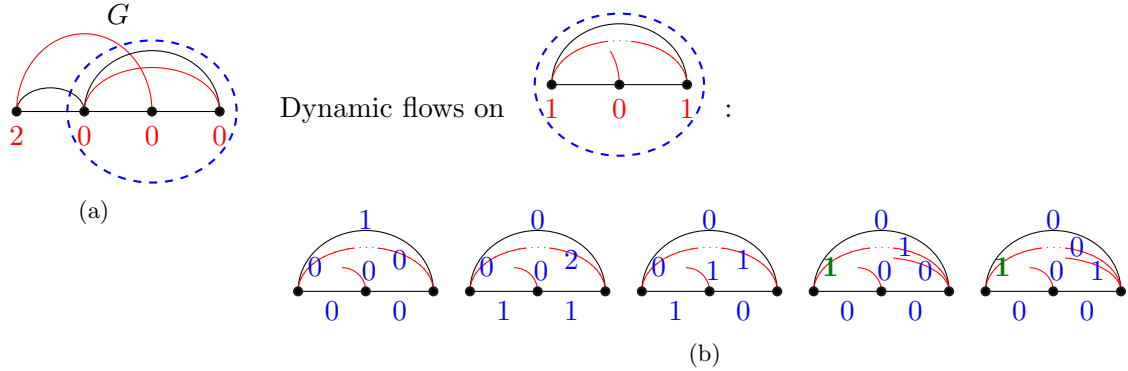
\end{example}

\noindent \textit{Proof of Theorem \ref{mm}.}
Recall   from the argument right before Definition~\ref{def:dynflows},  that $\vol(\F_G(2,0,\ldots,0))$ is the number of tuples $(T_2,\ldots,T_{n+1})$ of bipartite trees, each tree $T_{i+1}$ encoding a composition of $\#\mathcal{I}_{i+1}(G_i)-1$ with $\#\mathcal{O}_{i+1}(G_i)$ parts (where $G_i$ is the graph $(\cdots (G_{T_2}^{(2)})_{T_3}^{(3)}\cdots)_{T_i}^{(i)}$). We can encode the parts of each composition as dynamic integer flows on $G$ since by Lemma~\ref{lemG_Tdynflow} we will have
\[
\#\mathcal{O}_{i+1}(G) + \#\{\{\text{extra right half-edges $(\varnothing,i+1,+)$}\}\}=\#\mathcal{O}_{i+1}(G_i).
\] 
Next, we calculate the netflow on vertex $i+1$ of $G$:  by~\eqref{eqn:dynflow} we have
\begin{align*}
\lefteqn{a_{i+1} =}\\ 
&= \sum_{e\in \O^-_{i+1}} b(e) + \sum_{e=(i+1,\cdot,+)\in \O^+_{i+1}}b_{\ell}(e)+\sum_{e=(\cdot,i+1,+) \in \O^+_{i+1}} b_{r}(e) + \sum_{\text{extra $e'=(\varnothing,i+1,+)$}} b_r(e') - \sum_{e\in \I_{i+1}} b(e). 
\end{align*}
Where the contributions of the flows of outgoing edges and incoming edges are
\begin{equation}
\sum_{e\in \O^-_{i+1}} b(e) + \sum_{e=(i+1,\cdot,+)\in \O^+_{i+1}}b_{\ell}(e)+\sum_{e=(\cdot,i+1,+) \in \O^+_{i+1}} b_{r}(e) + \sum_{\text{extra $e'=(\varnothing,i+1,+)$}} b_r(e') = \#\mathcal{I}_{i+1}(G_i)-1, \label{eq:Doutdeg}
\end{equation}
\begin{equation}
\sum_{e\in \I_{i+1}} b(e) = \#\{\{\text{truly new edges $(\cdot,i+1,-)$}\}\}, \label{eq:Dindeg}
\end{equation}
where Equation~\eqref{eq:Doutdeg} follows since by repeated applications of Lemma~\ref{lemG_Tdynflow}: 
\[
(b_e)_{e\in \O_{i+1}(G) \cup \{\{\text{extra $e'=(\varnothing,i+1,+)$}\}\}}
\] 
is a composition of $\#I_{i+1}(G_i)-1$, and Equation~\eqref{eq:Dindeg} follows from Remark~\ref{compencoding} (ii). Then $a_{i+1}=\#\mathcal{I}_{i+1}(G_i)-1 -\#\{\{\text{truly new edges $(\cdot,i+1,-)$}\}\}$. Using Proposition~\ref{propG_T} this simplifies to
\begin{align*}
a_{i+1} &= \left(\#\mathcal{I}_{i+1}(G)+\#\{\{\text{truly new edges $(\cdot,i+1,-)$}\}\}-1\right)-\#\{\{\text{truly new edges $(\cdot,i+1,-)$}\}\}.
\end{align*}
So $a_{i+1}=\#\mathcal{I}_{i+1}(G)-1=indeg_G(i+1)-1$. Thus we have a map from $(G;(T_2,\ldots,T_{n+1}))$ to an integer dynamic ${\bf a}$-flow in $G$ where ${\bf a}=(0,d_2,d_3,\ldots,d_{n+1})$ for $d_i=indeg_G(i)-1$. See Figure~\ref{tD} for an example of this map.

Next we show this map is bijective by building its inverse. Given such an integer dynamic flow in $G$, we read off the flows on the edges of $\mathcal{O}_i(G)$ and the extra right positive half-edges $e'$ incident to $i$ for $i=2,\ldots,n+1$ in clockwise order. We obtain a weak composition of 
\[
N_i:=\sum_{e\in \mathcal{O}_i(G)} b(e) + \sum_{\text{ extra half-edges $e'=(\varnothing,i,+)$}} b_r(e'),
\]
 into $\#\mathcal{O}_i(G) + \#\{\text{extra right half-edges $(\varnothing,i,+)$}\}$ parts. Next, we encode these compositions as signed noncrossing trees. By by repeated applications of Lemma~\ref{lemG_Tdynflow}, we know that
\[
\#\mathcal{O}_{i+1}(G) + \#\{\text{extra right half-edges $(\varnothing,i+1,+)$}\}=\#\mathcal{O}_{i+1}(G_i),
\] 
and it is not hard to show by induction that $N_{i+1}=\#\mathcal{I}_{i+1}(G_i)-1$ where $G_i= (\cdots  (G_{T_2}^{(2)})_{T_3}^{(3)}\cdots)_{T_{i}}^{(i)}$. Thus $T_i$ encodes a composition of $\I_{i+1}(G_i)-1$ with $\#\O_{i+1}(G_i)$ parts. Therefore, we also have a map from an integer dynamic $\g$-flow in $G$ to a tuple $(G;(T_2,\ldots,T_{n+1}))$. 

It is easy to see that the two maps described above are inverses of each other. This shows the first map is the correspondence we desired.     
\qed

We end this section by giving the generating series of $K_G^{\text{dyn}}({\bf a})$ which we will use when we apply Theorem~\ref{mm} in Section~\ref{CRYAD}.

\subsection{A generating series for the dynamic Kostant partition function} \label{genseriesKdyn}

Next, using Definitions~\ref{def:dynflows} and \ref{defdkpf} we give the generating series of the dynamic Kostant partition function $K_G^{\text{dyn}}({\bf a})$.

\begin{proposition} \label{prop:Kdyngs}
The generating series of the dynamic Kostant partition function is
\begin{equation}\label{Kdyngs}
\sum_{{\bf a} \in \mathbb{Z}^{n+1}} K_G^{\text{dyn}}({\bf a}){\bf x}^{\bf a} = \prod_{(i,j,-) \in E(G)} (1-x_ix_j^{-1})^{-1} \prod_{(i,j,+) \in E(G)} (1-x_i-x_j)^{-1},
\end{equation}
where ${\bf x}^{\bf a}=x_1^{a_1}x_2^{a_2}\cdots x_{n+1}^{a_{n+1}}$.
\end{proposition}

\begin{proof}
By Definition~\ref{def:dynflows} of the integer dynamic flow, if the left half-edge $(i,\varnothing,+)$ of a positive edge $e=(i,j,+)$ has flow $k\in \N$ then we add $k$ extra right half-edges $(\varnothing,j,+)$ incident to $j$ besides the existing half-edge $(\varnothing,j,+)$. In this case the contribution to the generating series of the dynamic integer flows is $x_i^k(1-x_j)^{-k-1}$. Thus the total contribution to the generating series from $e=(i,j,+)$ in $G$ is
\begin{align*}
\sum_{k\geq 0} x_i^k(1-x_j)^{-k-1} &= (1-x_j)^{-1}(1-x_i(1-x_j)^{-1})^{-1}\\
&= (1-x_i-x_j)^{-1}.
\end{align*}
In addition, just as in \eqref{gsKA} the contributions of negative edges $e=(i,j,-)$ is $(1-x_ix_j^{-1})^{-1}$. Taking the product of these contributions for each of the edges of $G$ gives the stated generating series $\sum_{{\bf a} \in \mathbb{Z}^{n+1}} K_G^{\text{dyn}}({\bf a}){\bf x}^{\bf a}$.
\end{proof}

\begin{remark}
By assigning the possible integer flows to left half-edges, adding the appropriate number of right half-edges and correcting the netflow, it is possible to write the dynamic Kostant partition function $K_G^{\text{dyn}}(\g)$ as a finite sum of Kostant partition functions. For example for the graph $G$ in Figure~\ref{fig:dynflow}: $K_G^{\text{dyn}}(2,1,1) = K_{G_{(0)}}(2,1,1)+K_{G_{(1)}}(1,1,1)+K_{G_{(2)}}(0,1,1)=3+8+6$ where $G_{(i)}$, for $i=0,1,2$,  is obtained from $G$ by setting the flow  on the left half-edge $(1,\varnothing,+)$ to be $i$ and adding $i$ right half-edges $(\varnothing,4,+)$. This observation together with the piecewise quasipolynomiality of $K_G(\cdot)$ imply that $K_G^{\text{dyn}}(\g)$ is a sum of piecewise quasipolynomial functions. It would be interesting to study the chamber structure of $K_G^{\text{dyn}}$.
\end{remark}

\section{The volumes of the (signed) Chan-Robbins-Yuen polytopes}  \label{CRYAD}

When $H=K_{n+1}$, the complete graph on $n+1$ vertices, $\mathcal{F}_{K_{n+1}}(1,0,\ldots,0,-1)$ is also known as the Chan-Robbins-Yuen polytope $CRYA_n$ \cite{CR,CRY} (see Examples~\ref{mainexs} (iii)). Such polytope is a face of the Birkhoff polytope of all $n\times n$ doubly stochastic matrices. Zeilberger computed in \cite{Z} the volume of this polytope using the {\em Morris identity} \cite[Thm. 4.13]{WM}.  This polytope has drawn much attention with its combinatorial-looking volume $\prod_{i=1}^{n-2}Cat(i)$, and the lack of a combinatorial proof of this volume formula. In this section we study $CRYA_n$ and its type $C_{n+1}$ and $D_{n+1}$ generalizations.

\subsection{Chan-Robbins-Yuen polytope of type $A_n$} We reproduce an equivalent proof of Zeilberger's result using Theorem~\ref{sp}. First we mention the version of the identity used in \cite{Z} and a special value of it which gives a product of consecutive Catalan numbers. Then we use Theorem~\ref{sp} to show that the volume of the polytope reduces to this value of the identity.


\begin{lemma}[Morris Identity \cite{Z}] \label{morrisID}
For a positive integers $m$, $a$, and $b$, and positive half integers $c$, let 
\[
H(a,b,c; x_1,x_2,\ldots,x_m) := \prod_{i=1}^m x_i^{-a} (1-x_i)^{-b} \prod_{1\leq i<j\leq m} (x_j-x_i)^{-2c},
\]
and let $M_m(a,b,c)= CT_{x_m} \cdots CT_{x_1} H(a,b,c; x_1,x_2,\ldots,x_m)$, where $CT_{x_i}$ mean the constant term in the expansion of the variable $x_i$. Then

\begin{equation} \label{morrid}
M_m(a,b,c)  = \frac{1}{m!} \prod_{j=1}^{m-1} \frac{\Gamma(a+b+(m-1+j)c) \Gamma( c )}{\Gamma(a+jc+1)\Gamma(b+jc) \Gamma(c+jc)},
\end{equation}
where $\Gamma(\cdot)$ is a gamma function ($\Gamma(j)=(j-1)!$ when $j\in \mathbb{N}$).
\end{lemma}

Next, we give a special value of this identity.
\begin{corollary}[\cite{Z}] For the constant term $M_m(a,b,c)$ defined above, we have
\begin{equation} \label{moddidVal}
M_m(2,0,1/2)=M_m(1,1,1/2)=\prod_{k=1}^m Cat(k),
\end{equation}
where $Cat(k)=\frac{1}{k+1}\binom{2k}{k}$ is the $k$th Catalan number.
\end{corollary}

\begin{corollary}[\cite{Z}] For $n\geq 1$,  let $K_{n+1}$ be the complete graph on $n+1$ vertices. Then the volume of the flow polytope $\F_{K_{n+1}}(1,0,\ldots,0,-1)$ is
\[
\vol(\mathcal{F}_{K_{n+1}}(1,0,\ldots,0,-1)) = \prod_{k=0}^{n-2} Cat(k),
\]
where $Cat(k)=\frac{1}{k+1}\binom{2k}{k}$ is the $k$th Catalan number.
\end{corollary}

\begin{proof}
If $H=K_{n+1}$, by Theorem \ref{sp} we have that 
\begin{align*}
\vol(\F_{K_{n+1}}(1,0,\ldots,0,-1)) &= K_{K_{n+1}}(0,0,1,2,\ldots,n-2,-(n-2)(n-1)/2)\\
&= K_{K_{n-1}}(1,2,\ldots,n-2,-(n-2)(n-1)/2),
\end{align*}
where we reduced from $K_{n+1}$ to $K_{n-1}$ since the netflow on the first two vertices of $K_{n+1}$ is zero. Then from the generating series of the Kostant partition function \eqref{gsKA}: 
\begin{equation} \label{cryAgs1}
K_{K_{n-1}}(1,2,\ldots,n-2,-{\textstyle \binom{n-1}{2}}) = [x_1^1x_2^2\cdots x_{n-2}^{n-2}] \left.\prod_{1\leq i<j\leq n-1} (1-x_ix_j^{-1})^{-1}\right|_{x_{n-1}=1}
\end{equation}
where we have set $x_{n-1}=1$ since its power is determined by the power of the other variables. Since $1/(1-x_ix_j^{-1})= x_j/(x_j-x_i)$ then 
\begin{align}
 \vol \F_{K_{n+1}}(1,0,\ldots,0,-1)
 &= [x_1^1x_2^2\cdots x_{n-2}^{n-2}]\,\, x_1^0x_2^1x_3^2\cdots  x_{n-2}^{n-3}  
\prod_{i=1}^{n-2} (1-x_i)^{-1} \prod_{1\leq i<j \leq n-2} (x_j-x_i)^{-1} \nonumber \\ 
\label{cryAgs2} &= [x_1x_2\cdots x_{n-2} ]\prod_{i=1}^{n-2} (1-x_i)^{-1} \prod_{1\leq i<j \leq n-2} (x_j-x_i)^{-1}.
\end{align}
Since $[x] f(x)= CT_x \frac{1}{x} f(x)$ we get
\begin{equation} \label{cryAgs3}
\vol \F_{K_{n+1}}(1,0,\ldots,0,-1)
= CT_{x_{n-2}}CT_{x_{n-3}} \cdots CT_{x_1} \prod_{i=1}^{n-2}x_i^{-1} (1-x_i)^{-1} \prod_{1\leq i<j \leq n-2} (x_j-x_i)^{-1}.
\end{equation}
Note that the right-hand-side above is $M_{n-2}(1,1,1/2)$. Then by \eqref{moddidVal} the result follows.
\end{proof}

\begin{remark}
(i) Note that in this case of $H=K_{n+1}$, the multiset $\{\{\alpha_i\}\}$ of roots corresponding to the edges of $K_{n+1}$ are all the positive type $A_{n}$ roots, and the netflow vector $(1,0,\ldots,0,-1)$ is the highest root in type $A_{n}$. The volumes of $\F_{K_{n+1}}(\g)$ for generic positive roots in $A_n$ do not appear to have nice product formulas.
\noindent (ii) There is no combinatorial proof for the formula of the normalized volume of $\mathcal{F}_{K_{n+1}}(1,0,\ldots,0,-1)$. Another proof of this formula using residues was given by Baldoni and Vergne \cite{BV1,BVMorr}.
\end{remark}

\subsection{Volumes of Chan-Robbins-Yuen polytopes of type $C_n$ and type $D_n$.}


Recall from Examples~\ref{mainexs}~(iv) that $K^{D}_{n}$ is the complete signed graph on $n$ vertices (all edges of the form $(i,j,\pm)$ for $1\leq i<j\leq n$ corresponding to all the positive roots in type $D_{n}$), and   $CRYD_n=\mathcal{F}_{K_{n}^{D}}(2,0,\ldots,0)$ is an analogue of the Chan-Robbins-Yuen polytope. Next, using Theorem~\ref{mm} and Proposition~\ref{prop:Kdyngs} we express the volume of this polytope as the constant of a certain rational function. This is an analogue of~\eqref{cryAgs3}.

\begin{proposition}
Let $CRY D_n$ be the flow polytope $\F_{K_n^D}(2,0,\ldots,0)$ where $K_n^D$ is the complete signed graph with $n$ vertices (all edges of the form $(i,j,\pm)$, $1\leq i<j \leq n$). Then 
\begin{equation} \label{eqCRYDMorrislike}
\vol(CRY D_n) = CT_{x_{n-2}}\cdots CT_{x_1} \prod_{i=1}^{n-2}x_i^{-1} (1-x_i)^{-2}  \prod_{1\leq i<j \leq n-2} (x_j-x_i)^{-1}(1-x_j-x_i)^{-1} 
\end{equation}
\end{proposition}

\begin{proof}
By Theorem~\ref{mm} if $G=K^{D}_n$  we have that 
\[
\vol(\F_{K_n^D}(2,0,\ldots,0)) = K_{K_n^D}^{dyn}(0,0,1,2,\ldots,n-2).
\]
So by Proposition~\ref{prop:Kdyngs} and since the netflow on the first two vertices is zero this volume is given in terms of the generating series \eqref{Kdyngs} of $K^{dyn}_{K^D_n}$ by 
\[
\vol(\F_{K_n^D}(2,0,\ldots,0)) =  [x_3^1x_4^2\cdots x_n^{n-2}] CT_{x_2}CT_{x_1} \prod_{1\leq i<j\leq n} (1-x_ix_j^{-1})^{-1}(1-x_i-x_j)^{-1}.
\]
Then by plugging in $x_1=x_2=0$ and relabeling the variables $x_m\mapsto x_{m-2}$ on $\prod_{1\leq i<j\leq n} (1-x_ix_j^{-1})^{-1}(1-x_i-x_j)^{-1}$ gives:
\[
\vol(\F_{K_n^D}(2,0,\ldots,0)) = [x_1^1x_2^2\cdots x_{n-2}^{n-2}] \prod_{1\leq i<j\leq n-2} (1-x_ix_j^{-1})^{-1}(1-x_i-x_j)^{-1}\prod_{1\leq i\leq n-2} (1-x_i)^{-2}.
\]
In addition, just as we did with $CRY A_{n+1}$ in~\eqref{cryAgs1}-\eqref{cryAgs3} the above equation is equivalent to the desired expression:
\begin{equation*}  
\vol(\F_{K_n^D}(2,0,\ldots,0)) = CT_{x_{n-2}}CT_{x_{n-3}} \cdots CT_{x_1} \prod_{i=1}^{n-2}x_i^{-1} (1-x_i)^{-2}  \prod_{1\leq i<j \leq n-2} (x_j-x_i)^{-1}(1-x_j-x_i)^{-1}.
\end{equation*}
\end{proof}

We get the following values for $v_{n}=\vol(CRYD_n)$ either through counting integer dynamic flows (code available at \cite{Code}), or using \eqref{eqCRYDMorrislike}, or direct volume computation (using the {\tt Maple} package {\tt convex} \cite{MF} and code from Baldoni-Beck-Cochet-Vergne \cite{BBCV}):

\[
\begin{array}{r|r|r|r|r|r|r}
n & 2 & 3 & 4 & 5 & 6 & 7 \\ \hline
v_{n}& {1} & {2} &  {32} & {5120} & {9175040} & {197300060160}\\ \hline
\frac{v_{n}}{v_{n-1}} &  & 2^1\cdot {\bf 1} &  2^3\cdot {\bf 2} &  2^5\cdot {\bf 10} & 2^7\cdot {\bf 14} & 2^{9} \cdot {\bf 42}
\end{array}
\]
which suggests the following conjecture:
\begin{conjecture} \label{conjcryDC} Let $CRYD_n$ be the flow polytope $\mathcal{F}_{K_{n}^{D}}(2,0,\ldots,0)$ where $K^{D}_{n}$ is the complete signed graph with $n$ vertices (all edges of the form $(i,j,\pm)$, $1\leq i<j\leq n$). Then the normalized volume of $CRYD_n$ is
\[
\vol(CRYD_n) = 2^{(n-2)^2} \prod_{k=0}^{n-2} Cat(k).
\]
\end{conjecture} 

\begin{remark}
The right-hand-side of~\eqref{eqCRYDMorrislike} looks like an evaluation of the right-hand-side of the following generalization of the  Morris identity (Lemma~\ref{morrisID}):
\[
CT_{x_m}\cdots CT_{x_1} \prod_{i=1}^m x_i^{-a} (1-x_i)^{-b}\prod_{1\leq i<j\leq m}(x_j-x_i)^{-2c}(1-x_i-x_j)^{-2d}.
\] 
for positive integers $m, a,$ and $b$ and positive half integers $c$ and $d$. We were unable to find a formula in terms of $m, a,b,c$ for such a generalization.
\end{remark}

Finally, we very briefly consider the flow polytopes: (i) $\mathcal{F}_{K_{n}^C}(2,0,\ldots,0)$ where $K_{n}^C$ is the complete signed graph with loops $(i,i,+)$ corresponding to the type $C$ positive roots $2\ee_i$, (ii) $\mathcal{F}_{K_{n}^B}(2,0,\ldots,0)$ where $K_{n}^B$ is the complete signed graph with loops $(i,i,+)$ corresponding to the type $B$ positive root $\ee_i$, (iii) $\F_{K_n^C}(1,1,0,\ldots,0)$, and (iv) $\F_{K_n^B}(1,1,0,\ldots,0)$. These polytopes also appear to have interesting volumes:



\begin{conjecture} Let $K_n^D, K_n^B, K_n^C$ be the signed complete graphs whose edges correspond to the positive roots in type $D_n$, $B_n$ and $C_n$ as defined above then 
\begin{equation}
\vol \mathcal{F}_{K_{n}^C}(2,0,\ldots,0) = 2^{n-2}\cdot \vol(CRYD_n)
\end{equation}
and except for $n=2$ (where $\vol\mathcal{F}_{K_{n}^{D}}(2,0)=\vol\mathcal{F}_{K_{n}^{D}}(1,1)$),
\begin{equation}
\vol \mathcal{F}_{K_{n}^{\{B,C,D\}}}(2,0,\ldots,0) = 
2\cdot \vol \mathcal{F}_{K_n^{\{B,C,D\}}}(1,1,0,\ldots,0).
\end{equation}
\end{conjecture}


\bibliography{biblio-flowpoly}{}

\begin{thebibliography}{10}

\bibitem{BLV}
W.~Baldoni, J.~de~Loera, and M.~Vergne.
\newblock Counting integer flows in networks.
\newblock {\em Comput. Math.}, (4):277--314, 2004.

\bibitem{BV2}
W.~Baldoni and M.~Vergne.
\newblock {K}ostant partitions functions and flow polytopes.
\newblock {\em Transform. Groups}, 13(3-4):447--469, 2008.

\bibitem{BBCV}
W.~Baldoni-Silva, M.~Beck, C.~Cochet, and M.~Vergne.
\newblock Volume computation for polytopes and partition functions for
  classical root systems.
\newblock {\em Discrete Comput. Geom.}, 2005.
\newblock Maple worksheets:
  \url{http://www.math.jussieu.fr/~vergne/work/IntegralPoints.html}.

\bibitem{BV1}
W.~Baldoni-Silva and M.~Vergne.
\newblock Residues formulae for volumes and {E}hrhart polynomials of convex
  polytopes.
\newblock \href{http://arxiv.org/abs/math/0103097}{arXiv:math/0103097}, 2001.

\bibitem{BVMorr}
W.~Baldoni-Silva and M.~Vergne.
\newblock {\em Non-Commutative Harmonic Analysis}, volume 220 of {\em Progress
  in Mathematics}, chapter Morris identities and the total residue for a system
  of type $A_r$, pages 1--19.
\newblock Birkh\"auser, 2004.

\bibitem{BR}
M.~Beck and S.~Robins.
\newblock {\em Computing the Continuous Discretely: Integer-Point Enumeration
  in Polyhedra}.
\newblock Springer-Verlag, 2007.

\bibitem{CR}
C.S. Chan and D.P. Robbins.
\newblock On the volume of the polytope of doubly stochastic matrices.
\newblock {\em Experiment. Math.}, 8(3):291--300, 1999.

\bibitem{CRY}
C.S. Chan, D.P. Robbins, and D.S. Yuen.
\newblock On the volume of a certain polytope.
\newblock {\em Experiment. Math.}, 9(1):91--99, 2000.

\bibitem{CC}
C.~Cochet.
\newblock Vector partition functions and representation theory.
\newblock \href{http://arxiv.org/abs/math/0506159}{arXiv:0506159}, 2005.

\bibitem{DM}
W.~Dahmen and C.A. Micchelli.
\newblock The number of solutions to linear diophantine equations and
  multivariate splines.
\newblock {\em Trans. Amer. Math. Soc.}, 308:509--532, 1988.

\bibitem{DCP}
C.~De~Concini and C.~Procesi.
\newblock {\em Topics in Hyperplane Arrangements, Polytopes and Box Splines}.
\newblock Springer, 2011.

\bibitem{MF}
M.~Franz.
\newblock Convex, a maple package for convex geometry.
\newblock version 1.1 available at
  \url{http://www.math.uwo.ca/~mfranz/convex/}, 2009.

\bibitem{GJH}
G.J. Heckman.
\newblock Projections of orbits and asymptotic behavior of multiplicities for
  compact connected lie groups.
\newblock {\em Invent. Math.}, 67:333--356, 1982.

\bibitem{JH}
James~E. Humphreys.
\newblock {\em Introduction to Lie Algebras and Representation Theory}.
\newblock Springer-Verlag, 1972.

\bibitem{m2}
K.~M\'esz\'aros.
\newblock Root polytopes, triangulations, and the subdivision algebra, {I}.
\newblock {\em Trans. Amer. Math. Soc.}, 363(8):4359--4382, 2011.

\bibitem{m1}
K.~M\'esz\'aros.
\newblock Root polytopes, triangulations, and the subdivision algebra, {I}{I}.
\newblock {\em Trans. Amer. Math. Soc.}, 363(11):6111--6141, 2011.

\bibitem{Code}
K.~M\'esz\'aros and A.H. Morales.
\newblock supplementary code and data.
\newblock \url{http://sites.google.com/site/flowpolytopes/}, 2012.

\bibitem{WM}
W.G. Morris.
\newblock {\em Constant Term Identities for Finite and Affine Root Systems:
  Conjectures and Theorems}.
\newblock PhD thesis, University of Wisconsin-Madison, 1982.

\bibitem{p}
A.~Postnikov, 2010.
\newblock personal communication.

\bibitem{S}
R.P. Stanley.
\newblock Acyclic flow polytopes and {K}ostant's partition function.
\newblock Conference transparencies,
  \url{http://math.mit.edu/~rstan/trans.html}, 2000.

\bibitem{Str}
B.~Sturmfels.
\newblock On vector partition functions.
\newblock {\em J. Combin. Theory Ser. A}, 72(2):302--309, 1995.

\bibitem{Z}
D.~Zeilberger.
\newblock Proof of a conjecture of {C}han, {R}obbins, and {Y}uen.
\newblock {\em Electron. Trans. Numer. Anal.}, 9:147--148, 1999.

\end{thebibliography}
\bibliographystyle{plain}

\end{document}